\documentclass[10pt]{amsart}
\usepackage{amssymb}

\newtheorem{proposition}{Proposition}[section]
\newtheorem{lemma}[proposition]{Lemma}
\newtheorem{corollary}[proposition]{Corollary}
\newtheorem{theorem}[proposition]{Theorem}

\theoremstyle{definition}

\newcommand{\selabel}[1]{\label{se:#1}}
\newcommand{\seref}[1]{Section~\ref{se:#1}}

\def\<{\leqslant}
\def\>{\geqslant}
\def\a{\alpha}
\def\b{\beta}

\def\g{\gamma}

\def\O{\Omega}

\def\ol{\overline}

\def\t{\triangle}
\def\e{\varepsilon}
\def\oo{\infty}

\def\s{\sigma}

\def\ot{\otimes}
\def\ra{\rightarrow}

\date{}

\begin{document}
\title[Tensor products of modules over Drinfeld Doubles of Taft algebras]
{Indecomposable decomposition of tensor products of modules over Drinfeld Doubles of Taft algebras}
\author{Hui-Xiang Chen}
\address{School of Mathematical Science, Yangzhou University,
Yangzhou 225002, China}
\email{hxchen@yzu.edu.cn}
\author{Hassen Suleman Esmael Mohammed}
\address{School of Mathematical Science, Yangzhou University,
Yangzhou 225002, China}
\email{esmailhassan313@yahoo.com}
\author{Hua Sun}
\address{School of Mathematical Science, Yangzhou University,
Yangzhou 225002, China}
\email{997749901@qq.com}
\thanks{2010 {\it Mathematics Subject Classification}. 16E05, 16G99, 16T99}
\keywords{indecomposable module, tensor product, Hopf algebra, Drinfeld double}
\begin{abstract}
In this paper, we study the tensor product structure of category of finite dimensional
modules over Drinfeld doubles of Taft Hopf algebras.
Tensor product decomposition rules for all finite dimensional indecomposable modules
are explicitly given.
\end{abstract}
\maketitle

\section{\bf Introduction}\selabel{1}

Representations of a Hopf algebra (up to isomorphism) form a ring, called Green ring,
in which the multiplication is given by the tensor product over the base field, and this ring is a
commutative ring in the case of the Drinfeld double and any quasitriangular Hopf algebras.
The tensor product of representations is an important ingredient in the representation theory of Hopf
algebra and quantum groups. In particular, the decomposition of the tensor product of
indecomposable modules into direct sum of indecomposables has received enormous attention.

However, in general, very little is known about how a tensor product of
two indecomposable modules decomposes into a direct sum of indecomposable modules.
There are some results for the decompositions of tensor products of modules over a Hopf algebra or a quantum group.
Premet \cite{Premet1991} dealt with finite dimensional indecomposable restricted modules for restricted simple
3-dimensional Lie algebra over an algebraically closed field of characteristic $p>2$, and studied
the decomposition of tensor product of such modules.
Witherspoon \cite{With} studied the Drinfeld double of a finite dimensional group algebra
in positive characteristic. She proved that the Green ring of the Drinfeld double of a
group algebra decomposes as a product of ideals associated to some subgroups of the original
group. Cibils \cite{Cib} determined all the graded Hopf algebras on a cycle path coalgebra
(which are just equal to the generalized Taft algebras
(see \cite{ChenHuangYeZhang2004, HuangChenZhang2004, Rad1975, Ta})),
and consider the decomposition of the tensor product of two indecomposable modules (see also \cite{Gunn}).
Moreover, Cibils also computed the Green ring of the Sweedler 4-dimensional Hopf algebra by generators
and relations. Kondo and Saito gave the decomposition of tensor products of modules
over the restricted quantum universal enveloping algebra associated to $\mathfrak{sl}_2$
in \cite{ks}. Zhang, Wu, Liu and Chen \cite{ZWLC} studied the ring structures of the Grothendieck groups
of the Drinfeld doubles of the Taft algebras $H_n(q)$.
Recently, Chen, Van Oystaeyen and Zhang \cite{ChVOZh} computed the Green rings of Taft algebras
$H_n(q)$, using the decomposition of tensor products of modules given by Cibils \cite{Cib}.
Li and Zhang \cite{LiZhang} computed the Green rings of the generalized Taft algebras.
When $n=2$, the Taft algebra $H_2(-1)$ is exactly the Sweedler's 4-dimensional Hopf algebra $H_4$
(see \cite{Sw, Ta}). Chen \cite{Ch5} gave the decomposition of tensor products of modules
over $D(H_4)$ and described the Green ring of $D(H_4)$.

We defined a Hopf algebra $H_n(1,q)$ by generators and relations
in \cite{Ch1} (see the next section), which is isomorphic to the Drinfeld double of a Taft algebra $H_n(q)$.
We also determined all finite dimensional indecomposable modules over $H_n(1,q)$ in \cite{Ch2, Ch3, Ch4}.
Taft algebras belong to the class of biproduct of Nichols algebras as well as ``rank one nilpotent type"
algebras. The presentation of Drinfeld doubles of rank one pointed Hopf algebras by generators and relations
is given in \cite{KropRad2006}, and for general biproduct of Nichols algebras in
\cite{RadSch2008}. The Drinfeld doubles of Taft algebras are examples of liftings of quantum planes,
whose simple modules, projective covers, primitive idempotents, blocks and quivers
are described in \cite{ChinKrop2009}.
Erdmann, Green, Snashall and Taillefer \cite{EGST} studied the representations of the Drinfeld
double of the generalized Taft algebras, and determined the decompositions
of the tensor products of two simple modules. They also described the non-projective summands of
the tensor products of some other modules, but the projective summands of these tensor products
are not described. A natural question is how to determine the decomposition of tensor
product of two indecomposable modules over the Drinfeld doubles of the (generalized) Taft algebras.

In this paper, we investigate the indecomposable decompositions of the tensor products of indecomposable
modules over the Drinfeld doubles $H_n(1, q)$ of the Taft algebras $H_n(q)$ for $n>2$.
The paper is organized as follow. In \seref{2},
we recall the structure of $H_n(1,q)$, its relation with the small quantum groups
and the classification of the indecomposable modules over $H_n(1,q)$.
In \seref{3}, we investigate the tensor product of a simple module with an indecomposable module over $H_n(1, q)$,
and decompose such tensor products into a direct sum of indecomposable modules,
where the decompositions of the tensor products of simple modules are known from \cite{Ch2, EGST}.
In \seref{4}, we investigate the tensor product of an indecomposable projective module with
a non-simple indecomposable module, and decompose such tensor products into a direct sum of indecomposable modules.
In \seref{5}, we investigate the tensor products of non-simple non-projective
indecomposable modules, and decompose such tensor products into a direct sum of indecomposable modules.

\section{\bf Preliminaries}\selabel{2}

Throughout, we work over an algebraically closed field $k$. Unless
otherwise stated, all algebras, Hopf algebras and modules are
defined over $k$; all modules are left modules and finite dimensional;
all maps are $k$-linear; dim and $\otimes$ stand for ${\rm dim}_k$ and $\otimes_k$,
respectively. For the theory of Hopf algebras and quantum groups, we
refer to \cite{Ka, Mon, Sw}. For the representation theory of finite dimensional algebras,
we refer to \cite{ARS}. Let $\mathbb Z$ denote all integers, and ${\mathbb Z}_n={\mathbb Z}/n{\mathbb Z}$
for an integer $n$.

\subsection{Module categories and duality}\selabel{2.1}
~~

For a finite dimensional algebra $A$, let mod$A$ denote the category of finite dimensional $A$-modules.
For a module $M$ in mod$A$ and a nonnegative integer $s$, let $sM$ denote the
direct sum of $s$ copies of $M$. Then $sM=0$ if $s=0$.
Let $P(M)$ and $I(M)$ denote the projective cover and the injective envelope of $M$, respectively.
Let l$(M)$ denote the length of $M$,
and let ${\rm rl}(M)$ denote the Loewy length (=radical length=socle length)
of $M$.

Let $H$ be a finite dimensional Hopf algebra. Then mod$H$ is a monoidal category \cite{Ka, Mon}.
If $H$ is a quasitriangular Hopf algebra, then $M\ot N\cong N\ot M$ for any $H$-modules
$M$ and $N$. It is well known that the Drinfeld double $D(H)$ of a finite dimensional Hopf algebra
$H$ is always symmetric (see \cite{Lo, ObSch, Ra}).
For any module $M$ in mod$H$, the dual space $M^*={\rm Hom}(M, k)$
is also an $H$-module with the action given by
$$(h\cdot f)(m)=f(S(h)\cdot m), \ h\in H,\ f\in M^*,\ m\in M,$$
where $S$ is the antipode of $H$. It is well known that
$(M\ot N)^*\cong N^*\ot M^*$ for any $H$-modules $M$ and $N$.
If $H$ is quasitriangular, then $S^2$ is inner, and so $M^{**}\cong M$ for any
$M\in{\rm mod}H$ (see \cite{Lo}). In this case, this gives rise to a duality $(-)^*$ from ${\rm mod}H$ to itself.

\subsection{Drinfeld doubles of Taft algebras and small quantum groups}\selabel{1.2}
~~

The Drinfeld doubles of Taft Hopf algebras and their finite representations
were investigated in \cite{Ch1, Ch2, Ch3, Ch4}. The representations of pointed Hopf algebras
and their Drinfeld doubles were also studied in \cite{KropRad2009}.
The Drinfeld doubles of Taft algebras are closely related with small quantum groups.

First assume that $q\in k$ is an $n^{th}$ primitive root of unity, $n\>2$. The Taft Hopf algebra $H_n(q)$ is
generated by two elements $g$ and $h$ subject to the relations (see \cite{Ta}):
$$g^n=1,\quad\quad h^n=0,\quad\quad gh=qhg.$$
The coalgebra structure and the antipode are determined by
$$
\begin{array}{lll}
\t(g)=g\otimes g, & \t(h)=h\otimes g+1\otimes h, & \e(g)=1,\\
\e(h)=0, & S(g)=g^{-1}=g^{n-1}, &  S(h)=-q^{-1}g^{n-1}h.
\end{array}
$$
Note that ${\rm dim}H_n(q)=n^2$, and $\{g^ih^j|0\<i, j\<n-1\}$ forms a $k$-basis for $H_n(q)$.
When $n=2$, $H_2(q)$ is exactly the Sweedler 4-dimensional Hopf algebra.
The Drinfeld double $D(H_n(q))$ can be described as follows.

Let $p\in k$. Then one can define an $n^4$-dimensional Hopf algebra $H_n(p, q)$, which is
generated  as an algebra by $a, b, c$ and $d$ subject to the relations:
$$\begin{array}{lllll}
ba=qab,& db=qbd, & ca=qac,& dc=qcd,& bc=cb,\\
a^n=0, & b^n=1, &c^n=1,& d^n=0, & da-qad=p(1-bc).
\end{array}$$
The coalgebra structure and the antipode are given by
$$
\begin{array}{lll}
\t(a)=a\otimes b+1\otimes a, & \e(a)=0, & S(a)=-ab^{-1}=-ab^{n-1},\\
\t(b)=b\otimes b, & \e(b)=1, & S(b)=b^{-1}=b^{n-1},\\
\t(c)=c\otimes c,& \e(c)=1, & S(c)=c^{-1}=c^{n-1},\\
\t(d)=d\otimes c+1\otimes d,& \e(d)=0, & S(d)=-dc^{-1}=-dc^{n-1}.
\end{array}
$$
$H_n(p, q)$ has a $k$-basis $\{a^ib^jc^ld^k|0\<i, j, l, k\<n-1\}$, and is not semisimple.
If $p\neq 0$, then $H_n(p, q)$ is isomorphic to $D(H_n(q))$ as a Hopf algebra.
In particular, we have $H_n(p, q)\cong H_n(1, q)\cong D(H_n(q))$ for any $p\neq 0$.
For the details, the reader is directed to \cite{Ch1, Ch2}.
When $n=2$ and $p=0$, $H_2(0,q)$ is exactly the Hopf algebra $\ol{\mathcal A}$ in \cite{LiHu}.

Next assume that $q\in k$ is an $m^{th}$ primitive root of unity with $m>2$. Let $U_q$
be the quantum enveloping algebra $U_q(\mathfrak{sl}_2)$ of Lie algebra $\mathfrak{sl}_2$
described in \cite{Ka}.
Then $U_q$ is a Hopf algebra. Let $n=m$ if $m$ is odd, and $n=\frac{m}{2}$ if $m$ is even.
Let $I$ be the ideal of $U_q$ generated by $E^n$, $F^n$ and $K^n-1$. Then $I$ is a Hopf ideal of $U_q$,
and hence one gets a quotient Hopf algebra $\ol{U}_q:=U_q/I$, the small quantum group.
Note that $q^2$ is an $n^{th}$ primitive root of unity.
Hence one can form a Hopf algebra $H_n(1, q^2)$ as above. Then a straightforward verification shows
that there is a Hopf algebra epimorphism $\phi: H_n(1, q^2)\ra\ol{U}_q$ determined by
(see \cite[Proposition 4.5]{Andrea})
$$\phi(a)=E,\ \phi(b)=K,\ \phi(c)=K,\ \phi(d)=q^{-2}(q-q^{-1})FK.$$
Let $C$ be the group of central group-like elements in $H_n(1,q^2)$. Then
Ker$(\phi)=(kC)^+H_n(1,q^2)$. Moreover,
$H_n(1,q^2)\cong kC\ot H_n(1,q^2)/(kC)^+H_n(1,q^2)\cong kC\ot \ol{U}_q$ as algebras if $n$ is odd.
Chari and Premet in \cite{ChariAPremet} worked out all indecomposable
modules for $\ol{U}_q$ when $m$ is odd. Hence the indecomposable modules over $H_n(1,q^2)$ can be induced
from those over $\ol{U}_q$ for any odd $n$. Note that Suter worked out all indecomposable modules
for a slightly different version of $\ol{U}_q$ in \cite{Suter1994}.

\subsection{Indecomposable modules over $H_n(1,q)$}\selabel{1.3}
~~

Let $J:={\rm rad}(H_n(1,q))$ stand for the Jacobson radical of $H_n(1,q)$.
Then  $J^3=0$ by \cite[Corollary 2.4]{Ch4}. This means that the Loewy length
of $H_n(1,q)$ is 3. In order to study the tensor products of modules over $H_n(1,q)$,
we need first to give the structures of all finite dimensional indecomposable $H_n(1,q)$-modules.
We will follow the notations of \cite{Ch4}. Unless otherwise stated, all modules are modules over $H_n(1,q)$
in what follows.

From \cite{Ch4}, we know that the socle series and the radical series of an indecomposable
module coincide. We list all indecomposable modules according to the Loewy length.
There are $n^2$ simple modules up to isomorphism.

Simple modules: $V(l,r)$, $1\<l\<n$, $r\in{\mathbb Z}_n$. $V(l,r)$ has
a standard $k$-basis $\{v_i|1\<i\<l\}$ such that
\begin{equation*}
\begin{array}{ll}
av_i=\left\{
\begin{array}{ll}
v_{i+1},& 1\<i<l,\\
0,& i=l,\\
\end{array}\right.&
dv_i=\left\{
\begin{array}{ll}
0, & i=1,\\
\a_{i-1}(l)v_{i-1},& 1<i\<l,\\
\end{array}\right.\\
bv_i=q^{r+i-1}v_i,\ 1\<i\<l, &cv_i=q^{i-r-l}v_i,\ 1\<i\<l,\\
\end{array}
\end{equation*}
where $\a_i(l)=(i)_q(1-q^{i-l})$ for $1\<i<l\<n$.
The simple modules $V(n,r)$, $r\in{\mathbb Z}_n$, are projective and injective.

Projective modules of Loewy length 3: Let $P(l, r)$ be the  projective cover of
$V(l, r)$, $1\<l<n$, $r\in{\mathbb Z}_n$. Then $P(l, r)$ is the injective envelope of $V(l, r)$ as well.
$P(l, r)$ has a standard $k$-basis $\{v_i|1\<i\<2n\}$ such that
\begin{equation*}
av_i=\left\{
\begin{array}{ll}
v_{i+1}, &1\<i<n \mbox{ or }n+1\<i<2n,\\
0, & i=n \mbox{ or }2n,\\
\end{array}\right.
\end{equation*}
\begin{equation*}
\begin{array}{ll}
bv_i=\left\{
\begin{array}{ll}
q^{r+i-1}v_i, & 1\<i\<n,\\
q^{r+l+i-1}v_i, & n+1\<i\<2n,\\
\end{array}\right. &
cv_i=\left\{
\begin{array}{ll}
q^{i-l-r}v_i, & 1\<i\<n,\\
q^{i-r}v_i, & n+1\<i\<2n,\\
\end{array}\right.\\
\end{array}
\end{equation*}
\begin{equation*}
dv_i=\left\{
\begin{array}{ll}
q^{i-1}v_{2n-l+i-1}, & i=1 \text{ or } l+1,\\
q^{i-1}v_{2n-l+i-1}+\a_{i-1}(l)v_{i-1}, & 1<i\<l,\\
\a_{i-l-1}(n-l)v_{i-1}, & l+1<i\<n,\\
0,& i=n+1 \text{ or } 2n-l+1,\\
\a_{i-n-1}(n-l)v_{i-1}, & n+1<i\<2n-l,\\
\a_{i-2n+l-1}(l)v_{i-1}, & 2n-l+1<i\<2n.\\
\end{array}\right.
\end{equation*}
Moreover, we have (see \cite{Ch4})
$$\begin{array}{l}
{\rm soc}P(l, r)={\rm rad}^2P(l,r)\cong P(l, r)/{\rm rad}P(l, r)=P(l, r)/{\rm soc}^2P(l, r)\cong V(l, r),\\
{\rm soc}^2P(l, r)/{\rm soc}(P(l, r))={\rm rad}P(l, r)/{\rm rad}^2P(l, r)\cong 2V(n-l, r+l).\\
\end{array}$$

For non-isomorphic indecomposable modules with Loewy length 2,
we list them according to the lengths and the co-lengths of their socles.
We say that an indecomposable module $M$ with ${\rm rl}(M)=2$
is of $(s,t)$-type if l$(M/{\rm soc}(M))=s$ and l$({\rm soc}(M))=t$. By \cite{Ch4},
if $M$ is of $(s, t)$-type, then $s=t+1$, or $s=t$, or $s=t-1$.
Note that $M$ is a string module for $s=t+1$ and $s=t-1$; $M$ is a band module for $s=t$.

String modules: The indecomposable modules of $(s+1,s)$-type are given by the syzygy functor $\O$.
Let $V(l, r)$ be the simple modules given above, $1\<l<n$, $r\in{\mathbb Z}_n$.
Then the minimal projective resolutions of $V(l, r)$ are given by
$$
\cdots\ra4P(n-l, r+l)\ra3P(l, r)\ra2P(n-l, r+l)\ra P(l, r)\ra
V(l, r)\ra 0.$$
By these resolutions, one can describe the structure of $\O^sV(l, r)$, $s\>1$
(see \cite{Ch4}). The string module $\O^sV(l, r)$ is of $(s+1, s)$-type.
The indecomposable modules of $(s,s+1)$-type are given by the cosyzygy functor $\O^{-1}$.
For $1\<l<n$ and $r\in{\mathbb Z}_n$, the minimal injective resolutions of $V(l, r)$ are given by
$$
0\ra V(l, r)\ra P(l, r)\ra 2P(n-l, r+l)\ra 3P(l, r)\ra 4P(n-l, r+l)
\ra\cdots.$$
By these resolutions, one can describe the structure of $\O^{-s}V(l, r)$, $s\>1$
(see \cite{Ch4}). The string module $\O^{-s}V(l, r)$ is of $(s, s+1)$-type.

Let $1\<l<n$, $r\in{\mathbb Z}_n$ and $s\>1$. If $s$ is odd, then we have
$$\begin{array}{c}
{\rm soc}(\O^sV(l, r))\cong\O^{-s}V(l, r)/{\rm soc}(\O^{-s}V(l, r))
\cong sV(l, r),\\
{\rm soc}(\O^{-s}V(l, r))\cong\O^sV(l, r)/{\rm soc}(\O^sV(l, r))\cong(s+1)V(n-l, r+l).\\
\end{array}$$
If $s$ is even, then we have
$$\begin{array}{c}
{\rm soc}(\O^sV(l, r))\cong\O^{-s}V(l, r)/{\rm soc}(\O^{-s}V(l, r))\cong sV(n-l, r+l),\\
{\rm soc}(\O^{-s}V(l, r))\cong\O^sV(l, r)/{\rm soc}(\O^sV(l, r))\cong(s+1)V(l, r).\\
\end{array}$$

Band modules: The indecomposable modules of $(s, s)$-type can be described as follows.
Let ${\mathbb P}^1(k)$ be the projective 1-space over $k$.
${\mathbb P}^1(k)$ can be regarded as the set of all 1-dimensional
subspaces of $k^2$. Let $\oo$ be a symbol with $\oo\not\in k$ and
let $\overline k=k\cup\{\oo\}$. Then there is a bijection between
$\overline k$ and $\mathbb{P}^1(k)$: $\a\mapsto L(\a,1)$,
$\oo\mapsto L(1,0)$, where $\a\in k$ and $L(\a,\b)$ denotes the
1-dimensional subspace of $k^2$ with basis $(\a,\b)$ for any
$0\not=(\a,\b)\in k^2$. In the following, we regard
$\mathbb{P}^1(k)=\overline k$.

If $M$ is of $(s,s)$-type then $M\cong M_s(l,r,\eta)$, where $1\<l<n$, $r\in{\mathbb Z}_n$
and $\eta\in{\mathbb P}^1(k)$ (see \cite{Ch4}).
The indecomposable module $M_1(l,r,\oo)$, $1\<l<n$, $r\in{\mathbb Z}_n$, has a standard basis $\{v_1, v_2, \cdots, v_n\}$
such that
\begin{equation*}
\begin{array}{ll}
av_i=\left\{\begin{array}{ll}
0, & i=n-l \mbox{ or }n,\\
v_{i+1}, & \mbox{ otherwise },\\
\end{array}\right. &
dv_i=\left\{\begin{array}{ll}
v_n, & i=1,\\
\a_{i-1}(n-l)v_{i-1}, & 1<i\<n-l,\\
0, & i=n-l+1,\\
\a_{i-n+l-1}(l)v_{i-1}, & n-l+1<i\< n,\\
\end{array}\right.\\
bv_i=q^{r+l+i-1}v_i, & cv_i=q^{i-r}v_i.\\
\end{array}
\end{equation*}
The indecomposable module $M_1(l,r,\eta)$, $1\< l<n$, $r\in{\mathbb Z}_n$, $\eta\in k$,
has a standard basis $\{v_1, v_2, \cdots, v_n\}$ with the action given by
\begin{equation*}
\begin{array}{ll}
av_i=\left\{\begin{array}{ll}
v_{i+1} , & 1\<i<n,\\
0 , & i=n,\\
\end{array}\right. &
dv_i=\left\{\begin{array}{ll}
\eta q^lv_n, & i=1,\\
\a_{i-1}(n-l)v_{i-1}, & 1<i\<n-l,\\
0, & i=n-l+1,\\
\a_{i-n+l-1}(l)v_{i-1}, & n-l+1<i\< n,\\
\end{array}\right.\\
bv_i=q^{r+l+i-1}v_i, & cv_i=q^{i-r}v_i.\\
\end{array}
\end{equation*}
Then the band modules $M_s(l, r,\eta)$ are determined recursively by the almost split sequences
$$0\ra M_s(l, r,\eta)\ra M_{s-1}(l, r,\eta)\oplus M_{s+1}(l, r,\eta)\ra M_s(l, r,\eta)\ra 0,$$
where $s\>1$, $M_0(l, r,\eta)=0$, $1\<l<n$, $r\in{\mathbb Z}_n$ and $\eta\in{\mathbb P}^1(k)$
(see \cite{Ch3, Ch4}). $M_s(l, r,\eta)$ also can be constructed recursively by using pullback
(see \cite[pp. 2823-2824]{Ch4}). $M_s(l, r,\eta)$ is a submodule of $sP(l,r)$
and a quotient module of $sP(n-l, r+l)$, and there is an exact sequence
$$0\ra M_s(l, r,\eta)\hookrightarrow sP(l, r)\ra M_s(n-l, r+l, -\eta q^l)\ra 0.$$
Hence $\O M_s(l, r, \eta)\cong \O^{-1} M_s(l, r, \eta)\cong M_s(n-l, r+l, -\eta q^l)$.
Moreover, for any $1\<i<s$, $M_s(l, r,\eta)$ contains a unique
submodule of $(i,i)$-type, which is isomorphic to $M_i(l, r,\eta)$
and the quotient module of $M_s(l, r,\eta)$ modulo the submodule of
$(i,i)$-type is isomorphic to $M_{s-i}(l, r,\eta)$. Hence there is
an exact sequence of modules
$$0\ra M_i(l, r,\eta)\hookrightarrow M_s(l, r, \eta)\ra M_{s-i}(l, r, \eta)\ra 0.$$

Erdmann, Green, Snashall and Taillefer studied the representations of the Drinfeld double
$D(\Lambda_{n,d})$ of the generalized Taft algebras $\Lambda_{n,d}$ in \cite{EGST}.
In case $d=n$, $\Lambda_{n,n}$ is the $n^2$-dimensional Taft Hopf algebra. For this reason,
$\Lambda_{n,d}$ is called a generalized Taft algebra
in \cite{ChenHuangYeZhang2004, HuangChenZhang2004}.
Moreover, $D(\Lambda_{n,n})\cong H_n(1, q)$ as Hopf algebras.
Hence one also can get all indecomposable modules over $H_n(1,q)$ from
\cite{EGST}. In this case, $V(l,r)$ is the simple module $L(1-2r-l,r)$, and
the band modules $M_s(l,r,0)$ and $M_s(l,r,\oo)$ are string modules of even length in \cite{EGST}.

Throughout the following, let $n$ be a fixed positive integer with $n>2$,
and $q\in k$ an $n^{th}$ primitive root of unity.
Let $P(n, r)=V(n, r)$ and $\O^0V(l,r)=V(l,r)$ for all $1\<l<n$ and $r\in\mathbb{Z}_n$,
and let $\a\oo=\oo\a=\oo$ for any $0\neq\a\in k$. Let $\mathcal M$ denote the
category of finite dimensional modules over $H_n(1, q)$.

\section{\bf Tensor product of a simple module with a module}\selabel{3}

In this section, we investigate the tensor product of a simple module with
an indecomposable module. Throughout the following, unless otherwise stated,
a module means a module over $H_n(1,q)$, and an isomorphism means a module isomorphism.

Note that $M\ot N\cong N\ot M$ for any modules $M$ and $N$
since $H_n(1, q)$ is a quasitriangular Hopf algebra.
For any $t\in{\mathbb Z}$, let $c(t):=[\frac{t+1}{2}]$ be the integer part of $\frac{t+1}{2}$.
That is, $c(t)$ is the maximal integer with respect to $c(t)\<\frac{t+1}{2}$.
Then $c(t)+c(t-1)=t$.

\subsection{Tensor product of two simple modules}\selabel{3.1}
~~

The decomposition of the tensor product of two simple modules has been determined in \cite{Ch2, EGST}.
We gave the decomposition of the tensor product $V(l, r)\ot V(l', r')$ for $l+l'\<n+1$,
and described the socle of $V(l, r)\ot V(l', r')$ for $l+l'>n+1$ in \cite{Ch2}.
Erdmann, Green, Snashall and Taillefer described the decomposition of the tensor product of
any two simple modules for the Drinfeld double of the generalized Taft algebras $\Lambda_{n,d}$
in \cite{EGST}. Putting $d=n$ in \cite{EGST}, one can get the decomposition of $V(l, r)\ot V(l', r')$
for $l+l'>n+1$ (also for $l+l'\<n+1$).

{\bf Convention}: If $\oplus_{l\<i\<m}M_i$ is a term in a decomposition of a module,
then it disappears when $l>m$.
For instance, in the decomposition of the following Proposition \ref{3.1}(2),
the term $\oplus_{t+1\<i\leqslant l-1}V(l+l'-1-2i, r+r'+i)$ disappears
when $l'=n$, or equivalently $t=l-1$.

\begin{proposition}\label{3.1} Let $1\leqslant l\leqslant l'\leqslant n$ and $r, r'\in{\mathbb Z}_n$.\\
$(1)$ If $l+l'\leqslant n+1$, then
$V(l, r)\ot V(l', r')\cong \oplus_{i=0}^{l-1}V(l+l'-1-2i, r+r'+i)$. In particular,
$V(1, r)\ot V(l', r')\cong V(l', r+r')$ for all $1\leqslant l'\leqslant n$ and $r, r'\in{\mathbb Z}_n$.\\
$(2)$  If $t=l+l'-(n+1)\>0$, then
$$\begin{array}{rcl}
V(l, r)\ot V(l', r')&\cong&(\oplus_{i=c(t)}^tP(l+l'-1-2i, r+r'+i))\\
&&\oplus(\oplus_{t+1\<i\leqslant l-1}V(l+l'-1-2i, r+r'+i)).\\
\end{array}$$
\end{proposition}

\begin{proof} It follows from \cite[Theorem 3.1]{Ch2} and \cite[Theorem 4.1]{EGST}.
\end{proof}

By the Fundamental Theorem of Hopf modules (see \cite{Mon}),
$M\ot P$ is projective for any projective module $P$ and any module $M$.
Thus, one gets the following corollary.

\begin{corollary}\label{3.2}
The subcategory consisting of semisimple modules and projective modules in
$\mathcal M$ is a monoidal subcategory of $\mathcal M$.
\end{corollary}

\subsection{Tensor product of a simple module with a projective module}\selabel{3.1}
~~

In this subsection, we determine the tensor product $V(l, r)\ot P(l',r')$ of a simple module with an indecomposable
projective module. As pointed out in the last subsection, $V(l, r)\ot P(l',r')$ is projective, and so it is also injective.
Thus, it is enough to determine the socle of $V(l, r)\ot P(l',r')$.
If $M$ is a submodule of the socle of $V(l, r)\ot P(l',r')$, then $P(M)$ ($\cong I(M)$) is isomorphic to a submodule of $V(l, r)\ot P(l',r')$.
We will manage to find a submodule $U$ of the socle of $V(l, r)\ot P(l',r')$ such that $P(U)$ and
$V(l, r)\ot P(l',r')$ have the same dimension. In this case, $V(l, r)\ot P(l',r')\cong P(U)$.
In the following, we will also use the fact that if a projective module $P$ is isomorphic to a submodule
of a quotient module of a module $M$, then $P$ is isomorphic to a summand of $M$.

\begin{theorem}\label{3.3}
Let $1\leqslant l, l'<n$ and $r, r'\in{\mathbb Z}_n$. Assume that $l+l'\leqslant n$.
Let $l_1={\rm min}\{l, l'\}$. Then
$$\begin{array}{rcl}
V(l,r)\ot P(l',r')
&\cong&(\oplus_{i=0}^{l_1-1}P(l+l'-1-2i, r+r'+i))\\
&&\oplus(\oplus_{c(l+l'-1)\<i\<l-1}2P(n+l+l'-1-2i, r+r'+i)).\\
\end{array}$$
\end{theorem}

\begin{proof}
We first assume that $l\<l'$ and let $V=V(l, r)\ot P(l', r')$. Then
$V_1:=V(l, r)\ot{\rm soc}(P(l', r'))$ is a submodule of $V$.
Since ${\rm soc}(P(l', r'))\cong V(l', r')$, it follows from Proposition \ref{3.1}(1)
that $V_1\cong V(l, r)\ot V(l', r')\cong \oplus_{i=0}^{l-1}V(l+l'-1-2i, r+r'+i)$.
Hence $P(V_1)$ can be embedded into $V$ as a submodule.
Now we have $P(V_1)\cong\oplus_{i=0}^{l-1}P(V(l+l'-1-2i, r+r'+i))
\cong\oplus_{i=0}^{l-1}P(l+l'-1-2i, r+r'+i)$. Since $1\leqslant l+l'-1-2i\leqslant n-1$
for all $0\leqslant i\leqslant l-1$, dim$(P(l+l'-1-2i, i))=2n$, and so
dim$(P(V_1))=2nl={\rm dim}(V)$. This implies
$$V(l, r)\ot P(l', r')\cong\oplus_{i=0}^{l-1}P(l+l'-1-2i, r+r'+i).$$

Next, assume that $l'<l$. Applying $V(l, r)\ot$ to the exact sequence
$0\ra V(l',r')\ra P(l', r')\ra\O^{-1}V(l', r')\ra 0$,
one gets anther exact sequence
$$0\ra V(l,r)\ot V(l',r')\ra V(l,r)\ot P(l',r')\ra V(l,r)\ot\O^{-1}V(l',r')\ra 0.$$
Note that $l\<n-l'$ and $l+n-l'-(n+1)=l-l'-1\>0$.
By soc$(\O^{-1}V(l',r'))\cong 2V(n-l',r'+l')$ and Proposition \ref{3.1}(2), we have
$$\begin{array}{rl}
V(l,r)\ot{\rm soc}(\O^{-1}V(l',r'))&\cong 2V(l,r)\ot V(n-l',r'+l')\\
&\cong(\oplus_{i=c(l-l'-1)}^{l-l'-1}2P(l+n-l'-1-2i, r+r'+l'+i))\\
&\oplus(\oplus_{i=l-l'}^{l-1}2V(l+n-l'-1-2i, r+r'+l'+i))\\
&\cong(\oplus_{i=c(l+l'-1)}^{l-1}2P(n+l+l'-1-2i, r+r'+i))\\
&\oplus(\oplus_{i=l}^{l+l'-1}2V(n+l+l'-1-2i, r+r'+i)).\\
\end{array}$$
Since $\oplus_{i=c(l+l'-1)}^{l-1}2P(n+l+l'-1-2i, r+r'+i)$ is projective and injective,
it follows that there is an epimorphism
$$\phi: V:=V(l,r)\ot P(l',r')\ra \oplus_{i=c(l+l'-1)}^{l-1}2P(n+l+l'-1-2i, r+r'+i)$$
such that Ker$(\phi)$ contains a submodule isomorphic to $V(l,r)\ot V(l',r')$.
Hence $V={\rm Ker}(\phi)\oplus P$, where $P$ is a submodule of $V$ with
$P\cong\oplus_{i=c(l+l'-1)}^{l-1}2P(n+l+l'-1-2i, r+r'+i)$, and Ker$(\phi)$ contains a submodule
$V_1$ with $V_1\cong V(l,r)\ot V(l',r')$. By Proposition \ref{3.1}(1),
$V_1\cong\oplus_{i=0}^{l'-1}V(l+l'-1-2i, r+r'+i)$. Hence
soc$(V)={\rm soc}({\rm Ker}(\phi))\oplus{\rm soc}(P)\supseteq{\rm soc}(V_1)\oplus {\rm soc}(P)
\cong(\oplus_{i=0}^{l'-1}V(l+l'-1-2i, r+r'+i))
\oplus(\oplus_{i=c(l+l'-1)}^{l-1}2V(n+l+l'-1-2i, r+r'+i))=:U$.
Thus, $P(U)$ is isomorphic to a submodule of $V$.
Then a straightforward computation shows that dim$P(U)=2nl={\rm dim}(V)$, and so
$$\begin{array}{rl}
V(l, r)\ot P(l', r')\cong P(U)
\cong&(\oplus_{i=0}^{l'-1}P(l+l'-1-2i, r+r'+i))\\
&\oplus(\oplus_{i=c(l+l'-1)}^{l-1} 2P(n+l+l'-1-2i, r+r'+i)).\\
\end{array}$$
This completes the proof.
\end{proof}

\begin{corollary}\label{3.4}
Let $1\leqslant l\leqslant n$ and $r, r'\in{\mathbb Z}_n$. Then
$V(1, r)\ot P(l, r')\cong P(l, r+r').$
\end{corollary}

\begin{proof}
It is follows from Theorem \ref{3.3} for $1\leqslant l< n$, and Proposition {3.1}(1) for $l=n$.
\end{proof}

\begin{theorem}\label{3.5}
Let $1\leqslant l, l'\<n$ with $l'\neq n$ and $r, r'\in{\mathbb Z}_n$.
Assume that $t=l+l'-(n+1)\>0$ and let $l_1={\rm min}\{l, l'\}$. Then\\
$$\begin{array}{rl}
V(l, r)\ot P(l', r')
\cong&(\oplus_{i=c(t)}^t2P(l+l'-1-2i, r+r'+i))\\
&\oplus(\oplus_{t+1\<i\<l_1-1}P(l+l'-1-2i, r+r'+i))\\
&\oplus(\oplus_{c(l+l'-1)\<i\<l-1}2P(n+l+l'-1-2i, r+r'+i)).\\
\end{array}$$
\end{theorem}

\begin{proof}
By Proposition \ref{3.1}(1) and Corollary \ref{3.4}, we only need to consider
the case of $r=r'=0$.

First assume that $l\<l'$. Then $t<l-1$ by $l'<n$. We have an exact sequence
$$0\ra V(l, 0)\ot\O V(l',0)\ra V(l, 0)\ot P(l', 0)\ra V(l, 0)\ot V(l', 0)\ra 0.$$
By Proposition \ref{3.1}(2), $\oplus_{i=c(t)}^tP(l+l'-1-2i, i)$ is isomorphic to
a summand of $V(l,0)\ot V(l',0)$. Hence there is a module epimorphism
$$\phi: V:=V(l, 0)\ot P(l', 0)\ra \oplus_{i=c(t)}^tP(l+l'-1-2i, i)$$
such that ${\rm Ker}(\phi)$ contains a submodule isomorphic to $V(l, 0)\ot\O V(l',0)$.
Note that $V(l, 0)\ot\O V(l',0)\supseteq V(l, 0)\ot{\rm soc}(\O V(l',0))\cong V(l, 0)\ot V(l',0)$.
Thus, again by Proposition \ref{3.1}(2), an argument similar to Theorem \ref{3.3}
shows that soc$(V)$ contains a submodule $U$ with
$$\begin{array}{rl}
U\cong&{\rm soc}(V(l, 0)\ot V(l',0))\oplus(\oplus_{i=c(t)}^tV(l+l'-1-2i, i))\\
\cong&(\oplus_{i=c(t)}^t2V(l+l'-1-2i, i))\oplus(\oplus_{i=t+1}^{l-1}V(l+l'-1-2i, i)).\\
\end{array}$$
Thus, $P(U)$ is isomorphic to a submodule of $V$.
Then a straightforward computation shows that dim$P(U)=2nl={\rm dim}(V)$.
It follows that
$$V\cong P(U)\cong(\oplus_{i=c(t)}^t2P(l+l'-1-2i, i))
\oplus(\oplus_{i=t+1}^{l-1}P(l+l'-1-2i, i)).$$

Now assume that $l'<l$. We have two exact sequences
$$\begin{array}{c}
0\ra V(l,0)\ot V(l',0)\ra V(l,0)\ot P(l',0)\ra V(l,0)\ot\O^{-1}V(l',0)\ra 0,\\
0\ra 2V(l,0)\ot V(n-l',l')\ra V(l,0)\ot\O^{-1}V(l',0)\ra V(l,0)\ot V(l',0)\ra 0.\\
\end{array}$$
Note that $n-l'<l$ and $l+n-l'-(n+1)=l-l'-1\>0$. By Proposition \ref{3.1}(2),
$2V(l,0)\ot V(n-l',l')$ contains a summand isomorphic to
$$\oplus_{i=c(l-l'-1)}^{l-l'-1}2P(l+n-l'-1-2i, l'+i)\cong \oplus_{i=c(l+l'-1)}^{l-1}2P(n+l+l'-1-2i, i),$$
and $V(l,0)\ot V(l',0)$ contains a summand isomorphic to
$\oplus_{i=c(t)}^tP(l+l'-1-2i, i).$
It follows from the last exact sequence that $V(l,0)\ot\O^{-1}V(l',0)$ contains a projective summand $P$ with
$$P\cong(\oplus_{i=c(t)}^tP(l+l'-1-2i, i))\oplus(\oplus_{i=c(l+l'-1)}^{l-1}2P(n+l+l'-1-2i, i)).$$
Then from the former exact sequence above, an argument similar to Theorem \ref{3.3} shows that
soc$(V(l,0)\ot P(l',0))$ contains a submodule $U$ with $U\cong{\rm soc}(V(l,0)\ot V(l',0))\oplus{\rm soc}(P)$.
By Proposition \ref{3.1}(2), we have
$$\begin{array}{rl}
U\cong&(\oplus_{i=c(t)}^{l'-1}V(l+l'-1-2i, i))\oplus{\rm soc}(P)\\
\cong&(\oplus_{i=c(t)}^t2V(l+l'-1-2i, i))\oplus(\oplus_{t+1\<i\<l'-1}V(l+l'-1-2i, i))\\
&\oplus(\oplus_{i=c(l+l'-1)}^{l-1}2V(n+l+l'-1-2i, i)).\\
\end{array}$$
Then one can check that dim$P(U)={\rm dim}(V(l,0)\ot P(l',0))$, and so
$$\begin{array}{rl}
V(l,0)\ot P(l',0)\cong P(U)
\cong&(\oplus_{i=c(t)}^t2P(l+l'-1-2i, i))\\
&\oplus(\oplus_{t+1\<i\<l'-1}P(l+l'-1-2i, i))\\
&\oplus(\oplus_{i=c(l+l'-1)}^{l-1}2P(n+l+l'-1-2i, i)).\\
\end{array}$$
This completes the proof.
\end{proof}

\subsection{Tensor product of a simple module with a string module}\selabel{3.3}
~~

In this subsection, we determine the tensor product $V(l, r)\ot \O^{\pm m}V(l', r')$
of a simple module with a string module. By \cite[p.438]{EGST}, we have
$$V(l, r)\ot \O^{\pm m}V(l', r')\cong\O^{\pm m}(V(l,r)\ot V(l',r'))\oplus P$$
for some projective module $P$. Moreover, the first summand on the right side of the above isomorphism
can be easily determined by Proposition \ref{3.1}. But, the projective summand $P$ is not given there.
We will use the decomposition of the tensor products of
$V(l,r)$ with some composition factors of $\O^{\pm m}V(l', r')$ to find some
projective summands of $V(l, r)\ot \O^{\pm m}V(l', r')$, and then compare the dimensions
of these modules to determine the projective module $P$.
Note that $\O^{\pm m}P=0$ for any $m>0$ and projective module $P$.

\begin{proposition}\label{3.6}
Let $1\leqslant l, l'<n$ and $r, r'\in{\mathbb Z}_n$.
Assume that $l+l'\leqslant n$. Let $l_1={\rm min}\{l,l'\}$. Then for all $m\geqslant 0$, we have
$$\begin{array}{rl}
&V(l,r)\ot \O^{\pm m}V(l',r')\\
\cong&(\oplus_{i=0}^{l_1-1}\O^{\pm m}V(l+l'-1-2i, r+r'+i))\\
&\oplus(\oplus_{c(l+l'-1)\<i\<l-1}(m+\frac{1-(-1)^m}{2})P(n+l+l'-1-2i, r+r'+i)).\\
\end{array}$$
In particular, $V(1,r)\otimes \O^{\pm m}V(l',r')\cong \O^{\pm m}V(l',r+r')$.
\end{proposition}

\begin{proof}
As stated above, we have $V(l, r)\ot \O^{\pm m}V(l', r')\cong\O^{\pm m}(V(l,r)\ot V(l',r'))\oplus P$
for some projective module $P$. Then by Proposition \ref{3.1}(1), we have
$$\begin{array}{rl}
\O^{\pm m}(V(l,r)\ot V(l',r'))\cong&\O^{\pm m}(\oplus_{i=0}^{l_1-1}V(l+l'-1-2i, r+r'+i))\\
\cong&\oplus_{i=0}^{l_1-1}\O^{\pm m}V(l+l'-1-2i, r+r'+i).\\
\end{array}$$
Hence
$$V(l, r)\ot \O^{\pm m}V(l', r')\cong(\oplus_{i=0}^{l_1-1}\O^{\pm m}V(l+l'-1-2i, r+r'+i))\oplus P.$$
If $l\<l'$, then a straightforward computation shows that
${\rm dim}(V(l, r)\ot \O^{\pm m}V(l', r'))={\rm dim}(\oplus_{i=0}^{l-1}\O^{\pm m}V(l+l'-1-2i, r+r'+i))$,
which implies ${\rm dim}(P)=0$, and so $P=0$. Thus, the desired decomposition follows for $l\<l'$.

Now suppose that $l'<l$. We may assume that $m$ is odd since the proof is similar
for $m$ being even. Then we have two exact sequences
$$\begin{array}{c}
V(l,r)\ot\O^mV(l',r')\ra(m+1)V(l,r)\ot V(n-l',r'+l')\ra 0,\\
0\ra(m+1)V(l,r)\ot V(n-l',r'+l')\ra V(l,r)\ot \O^{-m}V(l',r').\\
\end{array}$$
Note that $l\<n-l'$ and $l+n-l'-(n+1)=l-l'-1\>0$. By Proposition \ref{3.1}(2),
the projective module $\oplus_{i=c(l-l'-1)}^{l-l'-1}(m+1)P(l+n-l'-1-2i, r+r'+l'+i)$
is a summand of $(m+1)V(l,0)\ot V(n-l',l')$, and so it is a summand of
$V(l,r)\ot \O^{\pm m}V(l',r')$. Then by Krull-Schmidt Theorem, we have
$$\begin{array}{rl}
&V(l, r)\ot \O^{\pm m}V(l', r')\\
\cong&(\oplus_{i=0}^{l'-1}\O^{\pm m}V(l+l'-1-2i, r+r'+i))\\
&\oplus(\oplus_{i=c(l-l'-1)}^{l-l'-1}(m+1)P(l+n-l'-1-2i, r+r'+l'+i))\oplus Q\\
\cong&(\oplus_{i=0}^{l'-1}\O^{\pm m}V(l+l'-1-2i, r+r'+i))\\
&\oplus(\oplus_{i=c(l+l'-1)}^{l-1}(m+1)P(n+l+l'-1-2i, r+r'+i))\oplus Q,\\
\end{array}$$
for some projective module $Q$. By a straightforward computation, one finds that dim$Q=0$,
and so $Q=0$. This completes the proof.
\end{proof}

\begin{proposition}\label{3.7}
Let $1\leqslant l, l'\<n$ with $l'\neq n$ and $r, r'\in{\mathbb Z}_n$.
Assume that $t=l+l'-(n+1)\geqslant 0$. Let $l_1={\rm min}\{l, l'\}$.
Then for all $m\geqslant 1$, we have\\
$$\begin{array}{rl}
&V(l,r)\ot \O^{\pm m}V(l',r')\\
\cong&(\oplus_{t+1\<i\<l_1-1}\O^{\pm m}V(l+l'-1-2i, r+r'+i))\\
&\oplus(\oplus_{i=c(t)}^t(m+\frac{1+(-1)^m}{2})P(l+l'-1-2i, r+r'+i))\\
&\oplus(\oplus_{c(l+l'-1)\<i\<l-1}(m+\frac{1-(-1)^m}{2})P(n+l+l'-1-2i, r+r'+i)).\\
\end{array}$$
\end{proposition}

\begin{proof}
By Proposition \ref{3.1}(1), Corollary \ref{3.4} and Proposition \ref{3.6},
we only need to consider the case of $r=r'=0$. Now by Proposition \ref{3.1}(2), we have
$$\begin{array}{rl}
V(l, 0)\ot \O^{\pm m}V(l', 0)\cong&\O^{\pm m}(V(l,0)\ot V(l',0))\oplus P\\
\cong&(\oplus_{t+1\<i\<l_1-1}\O^{\pm m}V(l+l'-1-2i, i))\oplus P\\
\end{array}$$
for some projective module $P$.
We assume that $m$ is odd since the proof is similar for $m$ being even.
Then we have two exact sequences
$$\begin{array}{rl}
0\ra mV(l,0)\ot V(l',0)&\ra V(l,0)\ot\O^mV(l',0)\\
&\ra(m+1)V(l,0)\ot V(n-l',l')\ra 0,\\
0\ra(m+1)V(l,0)\ot V(n-l',l')&\ra V(l,0)\ot \O^{-m}V(l',0)\\
&\ra mV(l,0)\ot V(l',0)\ra 0.\\
\end{array}$$
By Proposition \ref{3.1}(2), the projective module
$\oplus_{i=c(t)}^tmP(l+l'-1-2i, i)$ is isomorphic to a summand
of $mV(l,0)\ot V(l',0)$. If $l'<l$, then $l+n-l'-(n+1)=l-l'-1\>0$ and $n-l'<l$
by $l+l'\>n+1$. Again by Proposition \ref{3.1}(2), $(m+1)V(l,0)\ot V(n-l',l')$ contains a summand isomorphic to
$$\oplus_{i=c(l-l'-1)}^{l-l'-1}(m+1)P(l+n-l'-1-2i, l'+i)\cong\oplus_{i=c(l+l'-1)}^{l-1}(m+1)P(n+l+l'-1-2i, i)$$
in this case. Thus, $V(l,0)\ot\O^{\pm m}V(l',0)$ contains a summand isomorphic to
$$(\oplus_{i=c(t)}^tmP(l+l'-1-2i, i))\oplus(\oplus_{c(l+l'-1)\<i\<l-1}(m+1)P(n+l+l'-1-2i, i))$$
in any case. Then it follows from Krull-Schmidt Theorem that
$$\begin{array}{rl}
&V(l,0)\ot \O^{\pm m}V(l',0)\\
\cong&(\oplus_{t+1\<i\<l_1-1}\O^{\pm m}V(l+l'-1-2i, i))\\
&\oplus(\oplus_{i=c(t)}^tmP(l+l'-1-2i, r+r'+i))\\
&\oplus(\oplus_{c(l+l'-1)\<i\<l-1}(m+1)P(n+l+l'-1-2i, i))\oplus Q\\
\end{array}$$
for some projective module $Q$. Then by a tedious but standard computation, one gets that dim$Q=0$,
and so $Q=0$. This completes the proof.
\end{proof}

\subsection{Tensor product of a simple module with a band module}\selabel{3.4}
~~

In this subsection, we investigate the tensor product $M=V(l,r)\ot M_s(l',r',\eta)$ of a simple module
with a band module. Erdmann, Green, Snashall and Taillefer in \cite{EGST} showed that
any non-projective indecomposable summand of $M$ is a band module.
They described the module on an example with $s=1$ for the special case $n=d=6$,
but the decomposition for general case is not given there. By tensoring with $V(2,0)$,
we will determine $M$ by the induction on $l$.
For $l=1$ and $l=2$, we determine $M$ by using some standard basis and the duality $(-)^*$.
For the induction step, we use the following isomorphism (see Proposition \ref{3.1})
$$V(2,0)\ot V(l,r)\ot M_s(l',r',\eta)\cong V(l+1,r)\ot M_s(l',r',\eta)\oplus V(l-1,r+1)\ot M_s(l',r',\eta).$$
If the decompositions of $V(l,r)\ot M_s(l',r',\eta)$ and
$V(l-1,r+1)\ot M_s(l',r',\eta)$ are known, then the decomposition of the module on the left side
is known, which yields the decomposition of $V(l+1,r)\ot M_s(l',r',\eta)$.
Consequently, one gets the decomposition of $M$ for all $1\<l\<n$.

\begin{lemma}\label{3.8}
Let $1\<l<n$, $r, r'\in{\mathbb Z}_n$ and $\eta\in{\mathbb P}^1(k)$.
Then for all $s\>1$, $V(1,r)\ot M_s(l',r',\eta)\cong M_s(l,r+r',\eta)$.
\end{lemma}

\begin{proof}
It is similar to \cite[Lemma 3.2 and Proposition 3.4]{Ch5}.
\end{proof}

For a module $M$, let $M_{(r)}=\{m\in M|bm=q^rm\}$, $r\in{\mathbb Z}_n$.
Then it follows from \cite[Lemma 2.1]{Ch2} that
$M=M_{(0)}\oplus M_{(1)}\oplus\cdots\oplus M_{(n-1)}$
as vector spaces and $cM_{(r)}\subseteq M_{(r)}$ for all $r\in{\mathbb Z}_n$.
If $f: M\ra N$ is a module map, then $f(M_{(r)})\subseteq N_{(r)}$ for any $r\in{\mathbb Z}_n$.

\begin{lemma}\label{3.9}
Let $1\<l<n$, $r\in{\mathbb Z}_n$ and $s\>1$. Then there is a basis $\{v_{i,j}|1\<i\<n, 1\<j\<s\}$
in $M_s(l,r,\oo)$ such that
$$\begin{array}{ll}
av_{i,j}=\left\{\begin{array}{ll}
v_{i+1, j-1}, & i=n-l,\\
0, & i=n,\\
v_{i+1,j}, & \mbox{otherwise},\\
\end{array}\right. & bv_{i,j}=q^{r+l+i-1}v_{i,j},\\
\end{array}$$
$$\begin{array}{ll}
dv_{i,j}=\left\{\begin{array}{ll}
v_{n,j}, & i=1,\\
\a_{i-1}(n-l)v_{i-1, j}, & 1<i\<n-l,\\
0, & i=n-l+1,\\
\a_{i-n+l-1}(l)v_{i-1,j}, & n-l+1<i\<n,\\
\end{array}\right. & cv_{i,j}=q^{i-r}v_{i,j},\\
\end{array}$$
where $1\< i\< n$, $1\< j\<s$ and $v_{n-l+1,0}=0$.
\end{lemma}

\begin{proof}
We prove the lemma by the induction on $s$. For $s=1$, it follows from \seref{2}.
Now let $s\>2$ and $M=M_s(l, r, \oo)$. Then by \cite[Theorem 3.10(2)]{Ch4}, $M$ contains a unique submodule $N$
of $(s-1,s-1)$-type. Moreover, $N\cong M_{s-1}(l, r,\oo)$ and $M/N\cong M_1(l,r,\oo)$.
By the induction hypothesis, $N$ contains a basis $\{v_{i,j}|1\<i\<n, 1\<j\<s-1\}$ as stated in the lemma.
Define a subspace $L$ of $N$ by $L={\rm span}\{v_{i,j}|1\<i\<n, 1\<j\<s-2\}$ for $s>2$,
and $L=0$ for $s=2$. Then $L$ is obviously a submodule of $N$,
and $L\cong M_{s-2}(l,r,\oo)$ for $s>2$ by the induction hypothesis.
It follows from \cite[Theorem 3.10(2)]{Ch4}
that $M/L\cong M_2(l, r, \oo)$. Since $M/N\cong M_1(l,r,\oo)$,
$M/N$ contains a standard basis $\{x_1, x_2, \cdots, x_n\}$ as stated in \seref{2}.
Let $\pi: M\ra M/N$  be the canonical epimorphism. Since $x_1\in(M/N)_{(r+l)}$
and $x_{n-l+1}\in(M/N)_{(r)}$,
$x_1=\pi(u_1)$ and $x_{n-l+1}=\pi(u_{n-l+1})$ for some $u_1\in M_{(r+l)}$ and $u_{n-l+1}\in M_{(r)}$.
Obviously, $u_1\notin N$ and $u_{n-l+1}\notin N$.
By \cite[Lemma 2.2]{Ch2}, we have that $a^{l-1}M_{(r)}\subseteq M_{(r+l-1)}$ and
$dM_{(r+l)}\subseteq M_{(r+l-1)}$. From $dx_1=x_n$, one gets
$\pi(du_1)=\pi(a^{l-1}u_{n-l+1})$.
Hence $du_1-a^{l-1}u_{n-l+1}\in N\cap M_{(r+l-1)}=N_{(r+l-1)}$,
and so $du_1=a^{l-1}u_{n-l+1}+x$ for some $x\in N_{(r+l-1)}$.
By the action of $a$ on the basis of $N$ described above, one can see that
$a^{l-1}N_{(r)}=N_{(r+l-1)}$. Therefore, there is an element $y\in N_{(r)}$
such that $x=a^{l-1}y$, and consequently, $du_1=a^{l-1}(u_{n-l+1}+y)$.
By replacing $u_{n-l+1}$ with $u_{n-l+1}+y$, we may assume that $x=0$, i.e.,
$du_1=a^{l-1}u_{n-l+1}$. From $ax_{n-l}=0$ and $ax_i=x_{i+1}$ for $1\<i<n-l$, one gets
$\pi(a^{n-l}u_1)=a^{n-l}x_1=0$. Hence $a^{n-l}u_1\in N\cap M_{(r)}=N_{(r)}$.

Now let $u_i\in M$, $1\<i\<n$, be defined by $u_i=a^{i-1}u_1$ for $1\<i\<n-l$,
and  $u_i=a^{i-n+l-1}u_{n-l+1}$ for $n-l+1\<i\<n$.
Then $x_i=\pi(u_i)$ for all $1\<i\<n$. By the discussion for $M_s(l,r,\oo)$
in \seref{2}, one knows that
$du_{n-l+1}=0$. Since $\{v_{n-l+1,j}|1\<j\<s-1\}$
is a basis of $N_{(r)}$, we have $a^{n-l}u_1=\sum_{j=1}^{s-1}\a_jv_{n-l+1, j}$ for some
$\a_1, \a_2, \cdots, \a_{s-1}\in k$. If $\a_{s-1}=0$ then
$a^{n-l}u_1\in L$. In this case, $\{\ol{v_{i, s-1}}, \ol{u_i}|1\<i\<n\}$ is a basis
of $M/L$, where $\ol{v}$ denotes the image of $v\in M$ under the canonical epimorphism
$M\ra M/L$. Obviously, ${\rm span}\{\ol{v_{i, s-1}}|1\<i\<n\}$ is a submodules of $M/L$.
By the discussion for $M_s(l,r,\oo)$ in \seref{2} together with $du_1=a^{l-1}u_{n-l+1}$ and $du_{n-l+1}=0$,
it is straightforward to check that ${\rm span}\{\ol{u_i}|1\<i\<n\}$
is also a submodules of $M/L$. Moreover, $M/L={\rm span}\{\ol{v_{i, s-1}}|1\<i\<n\}\oplus {\rm span}\{\ol{u_i}|1\<i\<n\}$.
This is impossible since $M/L\cong M_2(l, r,\oo)$ is indecomposable. Hence $\a_{s-1}\neq 0$.
Now let
$$\begin{array}{c}
v_{i,s}=\a_{s-1}^{-1}(u_i-\sum_{1\<j\<s-2}\a_jv_{i,j+1}), \ 1\<i\<n,\\
\end{array}$$
where we regard $\sum_{1\<j\<s-2}\a_jv_{i,j+1}=0$ for $s=2$.
Then $v_{i,s}\in M_{(r+l+i-1)}\backslash N$. Hence
$\{v_{i,j}|1\<i\<n, 1\<j\<s\}$ is a basis of $M$.
Obviously, $cv_{i,s}=q^{i-r}v_{i,s}$ for all $1\<i\<n$,
$av_{n,s}=0$ and $dv_{n-l+1,s}=0$.
By \cite[Eq.(2.4)]{Ch2} and $au_n=0$, one can check that $du_i=\a_{i-1}(n-l)u_{i-1}$
for $1<i\<n-l$. Then a straightforward verification shows that
$\{v_{i,j}|1\<i\<n, 1\<j\<s\}$ is a desired basis of $M$.
\end{proof}

\begin{lemma}\label{3.10}
Let $1\<l<n$, $r\in{\mathbb Z}_n$, $\eta\in k$ and $s\>1$. Then there is a basis $\{v_{i,j}|1\<i\<n, 1\<j\<s\}$
in $M_s(l,r,\eta)$ such that
\begin{equation*}
\begin{array}{ll}
av_{i,j}=\left\{\begin{array}{ll}
v_{i+1,j} , & 1\<i<n,\\
0 , & i=n,\\
\end{array}\right. &
bv_{i,j}=q^{r+l+i-1}v_{i,j},\\
\end{array}
\end{equation*}
\begin{equation*}
\begin{array}{ll}
dv_{i,j}=\left\{\begin{array}{ll}
v_{n,j-1}+\eta q^lv_{n,j}, & i=1,\\
\a_{i-1}(n-l)v_{i-1,j}, & 1<i\<n-l,\\
0, & i=n-l+1,\\
\a_{i-n+l-1}(l)v_{i-1,j}, & n-l+1<i\< n,\\
\end{array}\right. &cv_{i,j}=q^{i-r}v_{i,j},\\
\end{array}
\end{equation*}
where $1\<i\<n$, $1\<j\<s$ and $v_{n,0}=0$.
\end{lemma}

\begin{proof}
It is similar to Lemma \ref{3.9}.
\end{proof}

\begin{lemma}\label{3.11}
Let $r, r'\in{\mathbb Z}_n$, $\eta\in{\mathbb P}^1(k)$ and $s\>1$. Then
$$V(2, r)\ot M_s(1, r', \eta)\cong M_s(2, r+r',\eta q^{-1}(2)_q)\oplus sV(n, r+r'+1).$$
\end{lemma}

\begin{proof}
By Proposition \ref{3.1}(1) and Lemma \ref{3.8}, we may assume that $r=r'=0$.
We only consider the case of $\eta\in k$ since the proof is similar for $\eta=\oo$.
Assume $\eta\in k$ and let $M=V(2, 0)\ot M_s(1,0,\eta)$. By the discussion in \seref{2},
there is a standard
basis $\{v_1, v_2\}$ in $V(2, 0)$ such that
$$\begin{array}{l}
av_1=v_2,\\
av_2=0,\\
\end{array} \
\begin{array}{l}
bv_1=v_1,\\
bv_2=qv_2,\\
\end{array}\
\begin{array}{l}
cv_1=q^{-1}v_1,\\
cv_2=v_2,\\
\end{array}\
\begin{array}{l}
dv_1=0,\\
dv_2=\a_1(2)v_1.\\
\end{array}$$
By Lemma \ref{3.10},
there is a standard basis $\{v_{i,j}|1\<i\<n, 1\<j\<s\}$
in $M_s(1,0,\eta)$ such that for all $1\<i\<n$ and $1\<j\<s$,
\begin{equation*}
\begin{array}{ll}
av_{i,j}=\left\{\begin{array}{ll}
v_{i+1,j} , & 1\<i<n,\\
0 , & i=n,\\
\end{array}\right. & bv_{i,j}=q^iv_{i,j},\\
dv_{i,j}=\left\{\begin{array}{ll}
v_{n,j-1}+\eta qv_{n,j}, & i=1,\\
\a_{i-1}(n-1)v_{i-1,j}, & 1<i\<n-1,\\
0, & i=n,\\
\end{array}\right. &cv_{i,j}=q^iv_{i,j},\\
\end{array}
\end{equation*}
where $v_{n,0}=0$. Hence $\{v_1\ot v_{i,j}, v_2\ot v_{i,j}|1\<i\<n, 1\<j\<s\}$
is a basis of $M$.

For any $1\< i\< n$ and $1\<j\<s$, define $u_{i,j}\in M$ by
$u_{1,j}=((2)_q)^{s-j}(v_1\ot v_{2,j}+(2)_qv_2\ot v_{1,j})$
and $u_{i,j}=a^{i-1}u_{1,j}$ for $i>1$. Then by Lemma \ref{3.10}, a straightforward
verification shows that
$N:={\rm span}\{u_{i,j}|1\<i\<n, 1\<j\<s\}$ is a submodule of $M$ and
$N\cong M_s(2,0, \eta q^{-1}(2)_q)$.

Since $M_s(1,0,\eta)/{\rm soc}(M_s(1,0,\eta))\cong sV(n-1,1)$, there is an
epimorphism from $M$ to $s(V(2,0)\ot V(n-1,1))$.
By Proposition \ref{3.1}(2), $V(n, 1)$ is a projective summand of $V(2,0)\ot V(n-1,1)$.
It follows that $M$ contains a submodule $U$ isomorphic to $sV(n,1)$.
Obviously, $N\cap U=0$. Therefore, $M=N\oplus U\cong M_s(2,0,\eta q^{-1}(2)_q)\oplus sV(n,1)$
by ${\rm dim}(N\oplus U)={\rm dim}(M)$.
\end{proof}

\begin{lemma}\label{3.12}
Let $1\leqslant l\leqslant n$ and $r\in{\mathbb Z}_n$.
Then $V(l,r)^*\cong V(l,1-l-r)$ and $P(l,r)^*\cong P(l,1-l-r)$.
If $1\leqslant l<n$, then $(\O^mV(l,r))^*\cong\O^{-m}V(l, 1-l-r)$ and
$(\O^{-m}V(l, r))^*\cong\O^{m}V(l, 1-l-r)$
for all $m\>1$.
\end{lemma}

\begin{proof} The first isomorphism is due to \cite[Theorem 4.3]{Andrea},
and the rest follow from an argument similar to \cite[Lemma 3.16]{Ch5}.
\end{proof}

\begin{lemma}\label{3.13}
Let $1\<l<n$, $r\in{\mathbb Z}_n$, $\eta\in{\mathbb P}^1(k)$ and $s\>1$. Then
$$M_s(l,r,\eta)^*\cong M_s(n-l, 1-r, -\eta q^l).$$
\end{lemma}

\begin{proof}
At first, by an argument similar to \cite[Theorem 4.3]{Andrea}, one can check that
$M_1(l,r,\eta)^*\cong M_1(n-l, 1-r, -\eta q^l)$ for $\eta=\oo$ and $\eta\in k$, respectively.

Now assume $s>1$. Then $M_s(l,r,\eta)^*$ is indecomposable.
By the structure of $M_s(l,r,\eta)$, we have an exact sequence
$0\ra sV(l,r)\ra M_s(l,r,\eta)\ra sV(n-l, r+l)\ra 0$.
Applying the duality $(-)^*$ to the above exact sequence and using Lemma \ref{3.12}, one gets
another exact sequence
$$0\ra sV(n-l,1-r)\ra M_s(l, r, \eta)^*\ra sV(l, 1-r-l)\ra 0.$$
By the classification of indecomposable modules stated in \seref{2}, one knows that
$M_s(l, r, \eta)^*\cong M_s(n-l, 1-r, \a)$ for some $\a\in{\mathbb P}^1(k)$.
On the other hand, there is an epimorphism $M_s(l, r, \eta)\ra M_1(l, r, \eta)$ by \cite[Theorem 3.10(2)]{Ch4}.
Then by applying the duality $(-)^*$,
one gets a monomorphism $M_1(n-l, 1-r, -\eta q^l)\ra M_s(n-l, 1-r, \a)$. Again by \cite[Theorem 3.10(2)]{Ch4},
$M_s(n-l, 1-r, \a)$ contains a unique submodule of $(1,1)$-type, which is isomorphic to
$M_1(n-l, 1-r, \a)$. Hence $M_1(n-l, 1-r, -\eta q^l)\cong M_1(n-l, 1-r, \a)$, which implies
$\a=-\eta q^l$ by \cite[Theorem 3.10(4)]{Ch4}. It follows that
$M_s(l, r, \eta)^*\cong M_s(n-l, 1-r, -\eta q^l)$.
\end{proof}

\begin{corollary}\label{3.14}
Let $r, r'\in{\mathbb Z}_n$, $\eta\in{\mathbb P}^1(k)$ and $s\>1$. Then\\
$$V(2, r)\ot M_s(n-1, r', \eta)\cong M_s(n-2, r+r'+1,\eta(2)_q)\oplus sV(n, r+r').$$
\end{corollary}

\begin{proof}
It is enough to show the corollary for $r=r'=0$.
By Lemma \ref{3.11}, we have an isomorphism
$V(2, -1)\ot M_s(1, 1, -\eta q^{-1})\cong M_s(2, 0,-\eta q^{-2}(2)_q)\oplus sV(n, 1)$.
Then by applying the duality $(-)^*$ to the isomorphism, it follows from Lemmas \ref{3.12}
and \ref{3.13} that
$V(2, 0)\ot M_s(n-1, 0, \eta)\cong M_s(n-2, 1,\eta(2)_q)\oplus sV(n, 0).$
\end{proof}

\begin{lemma}\label{3.15}
Let $1<l'<n-1$, $r, r'\in{\mathbb Z}_n$, $\eta\in{\mathbb P}^1(k)$ and $s\>1$. Then
$$\begin{array}{rl}
&V(2, r)\ot M_s(l', r', \eta)\\
\cong &M_s(l'+1, r+r',\eta q^{-1}\frac{(l'+1)_q}{(l')_q})
\oplus M_s(l'-1, r+r'+1, \eta q\frac{(l'-1)_q}{(l')_q}).\\
\end{array}$$
\end{lemma}

\begin{proof}
It is enough to show the lemma for $r=r'=0$.
We only prove the lemma for $\eta\in k$ since the proof is similar for $\eta=\oo$.

Assume $\eta\in k$ and let $M=V(2, 0)\ot M_s(l',0,\eta)$.
Let $\{v_1, v_2\}$ be the standard basis of $V(2,0)$ as stated in the proof of Lemma \ref{3.11},
and let $\{v_{i,j}|1\<i\<n, 1\<j\<s\}$ be the standard basis of $M_s(l',0,\eta)$
as given in Lemma \ref{3.10}. Then $M$ has a $k$-basis $\{v_1\ot v_{i,j}, v_2\ot v_{i,j}|1\<i\<n, 1\<j\<s\}$.

Now let $\b=\frac{q^{l'}}{(1-q)(l')_q}$ and $\g=\frac{(l')_q}{(l'-1)_q}$.
For $1\<i\<n$ and $1\<j\<s$, define $u_{i,j}\in M$ by
$u_{1,j}=\g^{j}(v_1\ot v_{1,j}+\eta q^{l'}\b v_2\ot v_{n,j}+\b v_2\ot v_{n,j-1})$
and $u_{i,j}=a^{i-1}u_{1,j}$ for $1<i\<n$.
Then by a standard computation, one can check that
$N_1:={\rm span}\{u_{i,j}|1\<i\<n, 1\<j\<s\}$ is a submodule of $M$ and
$N_1\cong M_s(l'-1, 1, \eta q\g^{-1})=M_s(l'-1, 1, \eta q\frac{(l'-1)_q}{(l')_q})$
by Lemma \ref{3.10}.

Next, let $\theta=\frac{(l')_q}{(l'+1)_q}$. For $1\<i\<n$ and $1\<j\<s$, define $w_{i,j}\in M$
by $w_{1,j}=\theta^j(v_1\ot v_{2,j}+(l'+1)_qv_2\ot v_{1,j})$
and $w_{i,j}=a^{i-1}w_{1,j}$ for $1<i\<n$. Then a straightforward verification shows that
$N_2:={\rm span}\{w_{i,j}|1\<i\<n, 1\<j\<s\}$
is a submodule of $M$ and $N_2\cong M_s(l'+1, 0, \eta q^{-1}\theta^{-1})
=M_s(l'+1,0, \eta q^{-1}\frac{(l'+1)_q}{(l')_q})$.

Finally, since ${\rm soc}(N_1)\cong sV(l'-1,1)$ and ${\rm soc}(N_2)\cong sV(l'+1,0)$,
the sum $N_1+N_2$ is direct. Then it follows from ${\rm dim}(M)={\rm dim}(N_1\oplus N_2)$ that
$$\begin{array}{c}
M=N_1\oplus N_2\cong M_s(l'-1, 1, \eta q\frac{(l'-1)_q}{(l')_q})\oplus
M_s(l'+1,0, \eta q^{-1}\frac{(l'+1)_q}{(l')_q}).\\
\end{array}$$
\end{proof}

\begin{theorem}\label{3.16}
Let $1\<l, l'<n$, $r,r'\in{\mathbb Z}_n$, $\eta\in{\mathbb P}^1(k)$ and $s\geqslant 1$. Assume
that $l+l'\leqslant n$ and let $l_1={\rm min}\{l, l'\}$. Then
$$\begin{array}{rl}
&V(l,r)\ot M_s(l',r',\eta)\\
\cong&(\oplus_{i=0}^{l_1-1}M_s(l+l'-1-2i, r+r'+i, \eta q^{2i-l+1}\frac{(l+l'-1-2i)_q}{(l')_q}))\\
&\oplus(\oplus_{c(l+l'-1)\<i\<l-1} sP(n+l+l'-1-2i, r+r'+i)).\\
\end{array}$$
\end{theorem}

\begin{proof}
It is enough to show the proposition for $r=r'=0$. We prove it by the induction on $l$.
For $l=1$ and $l=2$, it follows from Lemmas \ref{3.8}, \ref{3.11} and \ref{3.15}, respectively.
Now let $l>2$ and assume that the theorem holds for less $l$.

Case 1: $l\<l'$. In this case, $l-2<l-1<l'$. Hence by the induction hypothesis,
Lemma \ref{3.15} and Proposition \ref{3.1}(1), we have
$$\begin{array}{rl}
&V(2,0)\ot V(l-1,0)\ot M_s(l',0,\eta)\\
\cong& \oplus_{i=0}^{l-2}V(2,0)\ot M_s(l+l'-2-2i, i, \eta q^{2i-l+2}\frac{(l'+l-2-2i)_q}{(l')_q})\\
\cong&(\oplus_{i=0}^{l-2}M_s(l+l'-1-2i, i, \eta q^{2i-l+1}\frac{(l'+l-1-2i)_q}{(l')_q}))\\
&\oplus(\oplus_{i=0}^{l-2}M_s(l+l'-3-2i, i+1, \eta q^{2i-l+3}\frac{(l'+l-3-2i)_q}{(l')_q}))\\
\end{array}$$
and
$$\begin{array}{rl}
&V(2,0)\ot V(l-1,0)\ot M_s(l',0,\eta)\\
\cong& V(l,0)\ot M_s(l',0,\eta)\oplus V(l-2,1)\ot M_s(l',0,\eta)\\
\cong& V(l,0)\ot M_s(l',0,\eta)
\oplus(\oplus_{i=0}^{l-3} M_s(l+l'-3-2i,i+1,\eta q^{2i-l+3}\frac{(l+l'-3-2i)_q}{(l')_q})).\\
\end{array}$$
Thus, using Krull-Schmidt Theorem, one gets that
$$\begin{array}{c}
V(l,0)\ot M_s(l',0,\eta)\cong
\oplus_{i=0}^{l-1}M_s(l+l'-1-2i, i, \eta q^{2i-l+1}\frac{(l+l'-1-2i)_q}{(l')_q}).\\
\end{array}$$

Case 2: $l=l'+1$. In this case,
$l-2<l-1=l'$. Hence by the induction hypothesis,
and Lemmas \ref{3.11} and \ref{3.15}, we have
$$\begin{array}{rl}
&V(2,0)\ot V(l-1,0)\ot M_s(l',0,\eta)\\
\cong& \oplus_{i=0}^{l-2}V(2,0)\ot M_s(l+l'-2-2i, i, \eta q^{2i-l+2}\frac{(l'+l-2-2i)_q}{(l')_q})\\
\cong& V(2,0)\ot M_s(1, l'-1, \eta q^{l-2}\frac{1}{(l')_q})\\
&\oplus(\oplus_{i=0}^{l'-2}V(2,0)\ot M_s(l+l'-2-2i, i, \eta q^{2i-l+2}\frac{(l'+l-2-2i)_q}{(l')_q}))\\
\cong& M_s(2, l'-1, \eta q^{l-3}\frac{(2)_q}{(l')_q})\oplus sV(n, l')\\
&\oplus(\oplus_{i=0}^{l'-2}M_s(l+l'-1-2i, i, \eta q^{2i-l+1}\frac{(l'+l-1-2i)_q}{(l')_q}))\\
&\oplus(\oplus_{i=0}^{l'-2}M_s(l+l'-3-2i, i+1, \eta q^{2i-l+3}\frac{(l'+l-3-2i)_q}{(l')_q})).\\
\end{array}$$

Then an argument similar to Case 1 shows that
$$\begin{array}{rl}
&V(l,0)\ot M_s(l',0,\eta)\\
\cong&(\oplus_{i=0}^{l'-1}M_s(l+l'-1-2i, i, \eta q^{2i-l+1}\frac{(l+l'-1-2i)_q}{(l')_q}))
\oplus sV(n,l').\\
\end{array}$$

Case 3: $l=l'+2$. In this case, $l'=l-2<l-1$. Hence by the induction hypothesis (or Case 2),
Proposition \ref{3.1}(2) and Lemma \ref{3.15}, we have
$$\begin{array}{rl}
&V(2,0)\ot V(l-1,0)\ot M_s(l',0,\eta)\\
\cong&(\oplus_{i=0}^{l'-1}V(2,0)\ot M_s(l+l'-2-2i, i, \eta q^{2i-l+2}\frac{(l+l'-2-2i)_q}{(l')_q}))\\
&\oplus sV(2,0)\ot V(n,l')\\
\cong&(\oplus_{i=0}^{l'-1}M_s(l+l'-1-2i, i, \eta q^{2i-l+1}\frac{(l'+l-1-2i)_q}{(l')_q}))\\
&\oplus(\oplus_{i=0}^{l'-1}M_s(l+l'-3-2i, i+1, \eta q^{2i-l+3}\frac{(l'+l-3-2i)_q}{(l')_q}))
\oplus sP(n-1, l'+1).\\
\end{array}$$
Then an argument similar to Case 1 shows that
$$\begin{array}{rl}
&V(l,0)\ot M_s(l',0,\eta)\\
\cong&(\oplus_{i=0}^{l'-1}M_s(l+l'-1-2i, i, \eta q^{2i-l+1}\frac{(l+l'-1-2i)_q}{(l')_q}))
\oplus sP(n-1,l'+1).\\
\end{array}$$

Case 4: $l>l'+2$. In this case, $l-1>l-2>l'$. We assume that $l+l'$ is odd since the proof is
similar for $l+l'$ being even. Then $c(l+l'-1)=c(l+l'-2)$. By the induction hypothesis,
Lemma \ref{3.15} and Theorems \ref{3.3} and \ref{3.5}, we have
$$\begin{array}{rl}
&V(2, 0)\ot V(l-1,0)\ot M_s(l',0,\eta)\\
\cong&(\oplus_{i=0}^{l'-1}V(2, 0)\ot M_s(l+l'-2-2i, i, \eta q^{2i-l+2}\frac{(l+l'-2-2i)_q}{(l')_q}))\\
&\oplus sV(2, 0)\ot P(n-1, c(l+l'-2))\\
&\oplus(\oplus_{c(l+l')\<i\<l-2} sV(2, 0)\ot P(n+l+l'-2-2i, i))\\
\cong&(\oplus_{i=0}^{l'-1}M_s(l+l'-1-2i, i, \eta q^{2i-l+1}\frac{(l'+l-1-2i)_q}{(l')_q}))\\
&\oplus(\oplus_{i=0}^{l'-1}M_s(l+l'-3-2i, i+1, \eta q^{2i-l+3}\frac{(l'+l-3-2i)_q}{(l')_q}))\\
&\oplus 2sP(n, c(l+l'-1))\oplus sP(n-2, c(l+l'+1))\\
&\oplus(\oplus_{c(l+l'+1)\<i\<l-2} s(P(n+l+l'-1-2i, i)\oplus P(n+l+l'-3-2i, i+1))\\
\cong&(\oplus_{i=0}^{l'-1}M_s(l+l'-1-2i, i, \eta q^{2i-l+1}\frac{(l'+l-1-2i)_q}{(l')_q}))\\
&\oplus(\oplus_{i=0}^{l'-1}M_s(l+l'-3-2i, i+1, \eta q^{2i-l+3}\frac{(l'+l-3-2i)_q}{(l')_q}))\\
&\oplus(\oplus_{i=c(l+l'-1)}^{l-2} sP(n+l+l'-1-2i, i))\\
&\oplus(\oplus_{i=c(l+l'-1)}^{l-1} sP(n+l+l'-1-2i, i)).\\
\end{array}$$
Then an argument similar to Case 1 shows that
$$\begin{array}{rl}
&V(l,0)\ot M_s(l',0,\eta)\\
\cong&(\oplus_{i=0}^{l'-1}M_s(l+l'-1-2i, i, \eta q^{2i-l+1}\frac{(l+l'-1-2i)_q}{(l')_q}))\\
&\oplus(\oplus_{i=c(l+l'-1)}^{l-1} sP(n+l+l'-1-2i, i)).\\
\end{array}$$
This completes the proof.
\end{proof}

\begin{theorem}\label{3.17}
Let $1\leqslant l, l'\<n$ with $l'\neq n$, $r, r'\in{\mathbb Z}_n$, $s\geqslant 1$ and $\eta\in\mathbb{P}^1(k)$.
Assume that $t=l+l'-(n+1)\geqslant 0$. Let $l_1={\rm min}\{l,l'\}$. Then
$$\begin{array}{rl}
&V(l,r)\ot M_s(l',r',\eta)\\
\cong&(\oplus_{t+1\<i\<l_1-1}M_s(l+l'-1-2i,r+r'+i,\eta q^{2i-l+1}\frac{(l+l'-1-2i)_q}{(l')_q}))\\
&\oplus(\oplus_{i=c(t)}^tsP(l+l'-1-2i, r+r'+i))\\
&\oplus(\oplus_{c(l+l'-1)\<i\<l-1}sP(n+l+l'-1-2i, r+r'+i)).\\
\end{array}$$
\end{theorem}

\begin{proof}
It is enough to show the theorem for $r=r'=0$.
We prove it by the induction on $l$ for the three cases: $t=0$, $t=1$ and $t\>2$, respectively.
Note that $l'=t+n+1-l\geqslant t+1$ by $l\<n$, and $l\>2$ by $l'<n$ and $l+l'\>n+1$.
If $l=2$, then $l'=n-1$. In this case, the desired decomposition follows from Corollary \ref{3.14}.
Now assume that $l>2$.

Suppose $t=0$. If $l\<l'+2$, then
the desired decomposition follows from an argument similar to Theorem \ref{3.16}.
Now let $l>l'+2$. Then by Theorem \ref{3.16}, we have
$$\begin{array}{rl}
V(l-2,1)\ot M_s(l',0,\eta)
\cong&(\oplus_{i=0}^{l'-1}M_s(l+l'-3-2i, i+1, \eta q^{2i-l+3}\frac{(l+l'-3-2i)_q}{(l')_q}))\\
&\oplus(\oplus_{i=c(l+l'-3)}^{l-3} sP(n+l+l'-3-2i, i+1))\\
\cong&(\oplus_{i=1}^{l'}M_s(l+l'-1-2i, i, \eta q^{2i-l+1}\frac{(l+l'-1-2i)_q}{(l')_q}))\\
&\oplus(\oplus_{i=c(l+l'-1)}^{l-2} sP(n+l+l'-1-2i, i))\\
\end{array}$$
and
$$\begin{array}{rl}
V(l-1,0)\ot M_s(l',0,\eta)
\cong&(\oplus_{i=0}^{l'-1}M_s(l+l'-2-2i, i, \eta q^{2i-l+2}\frac{(l+l'-2-2i)_q}{(l')_q}))\\
&\oplus(\oplus_{i=c(l+l'-2)}^{l-2} sP(n+l+l'-2-2i, i)).\\
\end{array}$$
We may assume $l+l'$ is odd since the proof is similar when $l+l'$ is even.
Then $c(l+l'-2)=c(l+l'-1)=\frac{l+l'-1}{2}$. Hence by
Theorems \ref{3.3} and \ref{3.5}, Corollary \ref{3.14} and Lemma \ref{3.15},
one can check that
$$\begin{array}{rl}
&V(2,0)\ot V(l-1,0)\ot M_s(l',0,\eta)\\
\cong&V(2,0)\ot M_s(n-1, 0, \eta q^{2-l}\frac{(n-1)_q}{(l')_q})\\
&\oplus(\oplus_{1\<i\<l'-1}V(2,0)\ot M_s(l+l'-2-2i, i, \eta q^{2i-l+2}\frac{(l+l'-2-2i)_q}{(l')_q}))\\
&\oplus sV(2,0)\ot P(n-1, \frac{l+l'-1}{2})\\
&\oplus(\oplus_{c(l+l'-1)<i\<l-2} sV(2,0)\ot P(n+l+l'-2-2i, i))\\
\cong& sV(n,0)\oplus(\oplus_{1\<i\<l'-1}M_s(l+l'-1-2i, i, \eta q^{2i-l+1}\frac{(l+l'-1-2i)_q}{(l')_q}))\\
&\oplus(\oplus_{i=1}^{l'}M_s(l+l'-1-2i, i, \eta q^{2i-l+1}\frac{(l+l'-1-2i)_q}{(l')_q}))\\
&\oplus(\oplus_{i=c(l+l'-1)}^{l-2}sP(n+l+l'-1-2i, i))\\
&\oplus(\oplus_{i=c(l+l'-1)}^{l-1}sP(n+l+l'-1-2i, i)).\\
\end{array}$$
Thus, it follows from an argument similar to the proof of Theorem \ref{3.16} that
$$\begin{array}{rl}
V(l,0)\ot M_s(l',0,\eta)
\cong&(\oplus_{1\<i\<l'-1}M_s(l+l'-1-2i, i, \eta q^{2i-l+1}\frac{(l+l'-1-2i)_q}{(l')_q}))\\
&\oplus sV(n,0)\oplus(\oplus_{i=c(l+l'-1)}^{l-1}sP(n+l+l'-1-2i, i)).\\
\end{array}$$

For $t=1$, the proof is similar to the case of $t=0$.
Now suppose $t\>2$. If $l\<l'+2$, then the desired decomposition
follows from an argument similar to Theorem \ref{3.16}. Now let $l>l'+2$.
Then by the induction hypothesis, we have
$$\begin{array}{rl}
&V(l-2,1)\ot M_s(l',0,\eta)\\
\cong&(\oplus_{i=c(t-2)}^{t-2}sP(l+l'-3-2i, i+1))\\
&\oplus(\oplus_{i=t-1}^{l'-1}M_s(l+l'-3-2i,i+1,\eta q^{2i-l+3}\frac{(l+l'-3-2i)_q}{(l')_q}))\\
&\oplus(\oplus_{i=c(l+l'-3)}^{l-3}sP(n+l+l'-3-2i, i+1))\\
\end{array}$$
and
$$\begin{array}{rl}
V(l-1,0)\ot M_s(l',0,\eta)
\cong&(\oplus_{i=c(t-1)}^{t-1}sP(l+l'-2-2i, i))\\
&\oplus(\oplus_{i=t}^{l'-1}M_s(l+l'-2-2i,i,\eta q^{2i-l+2}\frac{(l+l'-2-2i)_q}{(l')_q}))\\
&\oplus(\oplus_{i=c(l+l'-2)}^{l-2}sP(n+l+l'-2-2i, i)).\\
\end{array}$$
In the following, we only consider the case that $t$ and $l+l'$ are both odd, since the proofs
are similar for the other cases. In this case, $c(t)=c(t-1)+1=\frac{t+1}{2}$ and
$c(l+l'-1)=c(l+l'-2)=\frac{l+l'-1}{2}$. By Proposition \ref{3.1}(2), Theorems \ref{3.3},
\ref{3.5} and Lemma \ref{3.15}, a straightforward computation shows that
$$\begin{array}{rl}
&V(2,0)\ot V(l-1,0)\ot M_s(l',0,\eta)\\
\cong&(\oplus_{i=c(t)}^{t-1}sP(l+l'-1-2i, i))
\oplus(\oplus_{i=c(t)}^{t}sP(l+l'-1-2i, i))\\
&\oplus(\oplus_{i=t}^{l'-1} M_s(l+l'-1-2i,i,\eta q^{2i-l+1}\frac{(l+l'-1-2i)_q}{(l')_q}))\\
&\oplus(\oplus_{i=t+1}^{l'}M_s(l+l'-1-2i,i,\eta q^{2i-l+1}\frac{(l+l'-1-2i)_q}{(l')_q}))\\
&\oplus(\oplus_{i=c(l+l'-1)}^{l-2}sP(n+l+l'-1-2i, i))\\
&\oplus(\oplus_{i=c(l+l'-1)}^{l-1}sP(n+l+l'-1-2i, i)).\\
\end{array}$$
Then by an argument similar to the proof of Theorem \ref{3.16}, one gets that
$$\begin{array}{rl}
V(l,0)\ot M_s(l',0,\eta)
\cong&(\oplus_{i=c(t)}^{t}sP(l+l'-1-2i, i))\\
&\oplus(\oplus_{t+1\<i\<l'-1}M_s(l+l'-1-2i,i,\eta q^{2i-l+1}\frac{(l+l'-1-2i)_q}{(l')_q}))\\
&\oplus(\oplus_{i=c(l+l'-1)}^{l-1}sP(n+l+l'-1-2i, i)).\\
\end{array}$$
This completes the proof.
\end{proof}

\section{\bf Tensor product of a projective module
with a non-simple module}\selabel{4}

In this section, we investigate the tensor product $P\ot M$ of a non-simple projective
indecomposable module $P$ with a non-simple indecomposable module $M$. Since $P\ot N$ is projective
for any module $N$, $P\ot M$ is isomorphic to the direct sum of all the tensor products $P\ot V$, where $V$
ranges all composition factors of $M$. However, the decompositions of all such tensor products $P\ot V$
are known by Theorems \ref{3.3} and \ref{3.5}, which gives rise to the decomposition of $P\ot M$.

\begin{proposition}\label{4.1}  Let $1\<l\<l'<n$ and $r, r'\in{\mathbb Z}_n$. Assume that $l+l'\<n$.
Then for all $m\>1$, we have
$$\begin{array}{rl}
&P(l, r)\ot\O^{\pm m}V(l', r')\\
\cong &(\oplus_{i=0}^{l-1}(m+\frac{1+(-1)^m}{2})P(l+l'-1-2i, r+r'+i))\\
&\oplus(\oplus_{i=l'}^{l+l'-1}(m+\frac{1-(-1)^m}{2})P(n+l+l'-1-2i, r+r'+i))\\
&\oplus(\oplus_{c(l+l'-1)\<i\<l'-1}2(m+\frac{1+(-1)^m}{2})P(n+l+l'-1-2i, r+r'+i))\\
&\oplus(\oplus_{1\<i\<c(n-l-l')}2(m+\frac{1-(-1)^m}{2})P(l+l'-1+2i, r+r'-i)).\\
\end{array}$$
\end{proposition}

\begin{proof}
By Corollary \ref{3.4} and Proposition \ref{3.6}, it is enough to show the proposition for $r=r'=0$.
We may assume that $m$ is odd since the proof is similar when $m$ is even.
In this case, there are two exact sequences
$$\begin{array}{c}
0\ra mV(l', 0)\ra\O^mV(l', 0)\ra(m+1)V(n-l', l')\ra 0,\\
0\ra (m+1)V(n-l', l')\ra\O^{-m}V(l', 0)\ra mV(l', 0)\ra 0.\\
\end{array}$$
Applying $P(l, 0)\ot$ to the above sequences, one gets the following exact sequences
$$\begin{array}{rl}
0\ra mP(l, 0)\ot V(l', 0)\!\!\!&\ra P(l, 0)\ot\O^mV(l', 0)\\
&\ra(m+1)P(l, 0)\ot V(n-l', l')\ra 0,\\
0\ra(m+1)P(l, 0)\ot V(n-l', l')\!\!\!&\ra P(l, 0)\ot\O^{-m}V(l', 0)\\
&\ra mP(l, 0)\ot V(l', 0)\ra 0.\\
\end{array}$$
They are split since $P(l, 0)\ot V(l', 0)$ and $P(l, 0)\ot V(n-l', l')$ are both projective.
By Theorem \ref{3.3}, we have
$$\begin{array}{rl}
P(l,0)\ot V(l',0)
\cong&(\oplus_{i=0}^{l-1}P(l+l'-1-2i, i))\\
&\oplus(\oplus_{c(l+l'-1)\<i\<l'-1}2P(n+l+l'-1-2i, i)).\\
\end{array}$$
By $1\<l\<l'<n$ and $l+l'\<n$, one knows that $1\<l\<n-l'<n$ and $l+(n-l')\<n$.
Hence similarly, we have
$$\begin{array}{rl}
P(l,0)\ot V(n-l',l')
\cong&(\oplus_{i=0}^{l-1}P(l+n-l'-1-2i, l'+i))\\
&\oplus(\oplus_{c(l+n-l'-1)\<i\<n-l'-1}2P(2n+l-l'-1-2i, l'+i))\\
\cong&(\oplus_{i=l'}^{l+l'-1}P(n+l+l'-1-2i, i))\\
&\oplus(\oplus_{1\<i\<c(n-l-l')}2P(l+l'-1+2i, -i)).\\
\end{array}$$
It follows that
$$\begin{array}{rl}
&P(l, 0)\ot\O^mV(l', 0)\cong  P(l, 0)\ot\O^{-m}V(l', 0)\\
\cong&(\oplus_{i=0}^{l-1}mP(l+l'-1-2i, i))\\
&\oplus(\oplus_{c(l+l'-1)\<i\<l'-1}2mP(n+l+l'-1-2i, i))\\
&\oplus(\oplus_{i=l'}^{l+l'-1}(m+1)P(n+l+l'-1-2i, i))\\
&\oplus(\oplus_{1\<i\<c(n-l-l')}2(m+1)P(l+l'-1+2i, -i)).\\
\end{array}$$
\end{proof}

\begin{corollary}\label{4.2}  Let $2\<l\<l'<n$ and $r, r'\in{\mathbb Z}_n$. Assume that
$t=l+l'-(n+1)\>0$. Then for all $m\>1$,
$$\begin{array}{rl}
&P(l, r)\ot\O^{\pm m}V(l', r')\\
\cong&(\oplus_{i=c(t)}^t2(m+\frac{1+(-1)^m}{2})P(l+l'-1-2i, r+r'+i))\\
&\oplus(\oplus_{i=t+1}^{l-1}(m+\frac{1+(-1)^m}{2})P(l+l'-1-2i, r+r'+i))\\
&\oplus(\oplus_{i=l'}^{n-1}(m+\frac{1-(-1)^m}{2})P(n+l+l'-1-2i, r+r'+i))\\
&\oplus(\oplus_{c(l+l'-1)\<i\<l'-1}2(m+\frac{1+(-1)^m}{2})P(n+l+l'-1-2i, r+r'+i)).\\
\end{array}$$
\end{corollary}

\begin{proof}
By Theorem \ref{3.5}, we have
$$\begin{array}{rl}
P(l, 0)\ot V(l', 0)
\cong&(\oplus_{i=c(t)}^t2P(l+l'-1-2i, i))\\
&\oplus(\oplus_{i=t+1}^{l-1}P(l+l'-1-2i, i))\\
&\oplus(\oplus_{c(l+l'-1)\<i\<l'-1}2P(n+l+l'-1-2i, i)).\\
\end{array}$$
By $2\<l\<l'<n$ and $l+l'\>n+1$, we have $1\<n-l'<l<n$ and $l+(n-l')\<n$.
Hence by Theorem \ref{3.3}, we have
$$\begin{array}{rl}
P(l,0)\ot V(n-l',l')
\cong&\oplus_{i=0}^{n-l'-1}P(n+l-l'-1-2i, l'+i)\\
\cong&\oplus_{i=l'}^{n-1}P(n+l+l'-1-2i, i).\\
\end{array}$$
Then the corollary follows from the proof of Proposition \ref{4.1}.
\end{proof}

\begin{corollary}\label{4.3}  Let $1\<l'<l<n$ and $r, r'\in{\mathbb Z}_n$. Assume that $l+l'\<n$.
Then for all $m\>1$,
$$\begin{array}{rl}
&P(l, r)\ot\O^{\pm m}V(l', r')\\
\cong &(\oplus_{i=0}^{l'-1}(m+\frac{1+(-1)^m}{2})P(l+l'-1-2i, r+r'+i))\\
&\oplus(\oplus_{i=l}^{l+l'-1}(m+\frac{1-(-1)^m}{2})P(n+l+l'-1-2i, r+r'+i))\\
&\oplus(\oplus_{i=c(l+l'-1)}^{l-1}2(m+\frac{1-(-1)^m}{2})P(n+l+l'-1-2i, r+r'+i))\\
&\oplus(\oplus_{1\<i\<c(n-l-l')}2(m+\frac{1-(-1)^m}{2})P(l+l'-1+2i, r+r'-i)).\\
\end{array}$$
\end{corollary}

\begin{proof}
It is similar to Corollary \ref{4.2}.
\end{proof}

\begin{corollary}\label{4.4}  Let $2\<l'<l<n$ and $r, r'\in{\mathbb Z}_n$. Assume that
$t=l+l'-(n+1)\>0$. Then for all $m\>1$,
$$\begin{array}{rl}
&P(l, r)\ot\O^{\pm m}V(l', r')\\
\cong&(\oplus_{i=c(t)}^{t}2(m+\frac{1+(-1)^m}{2})P(l+l'-1-2i, r+r'+i))\\
&\oplus(\oplus_{i=t+1}^{l'-1}(m+\frac{1+(-1)^m}{2})P(l+l'-1-2i, r+r'+i))\\
&\oplus(\oplus_{i=l}^{n-1}(m+\frac{1-(-1)^m}{2})P(n+l+l'-1-2i, r+r'+i))\\
&\oplus(\oplus_{i=c(l+l'-1)}^{l-1}2(m+\frac{1-(-1)^m}{2})P(n+l+l'-1-2i, r+r'+i)).\\
\end{array}$$
\end{corollary}

\begin{proof}
It is similar to Corollary \ref{4.2} by using Theorem \ref{3.5}.
\end{proof}

\begin{proposition}\label{4.5}
Let $1\<l\<l'<n$ and $r, r'\in{\mathbb Z}_n$. Assume $l+l'\<n$.
Then
$$\begin{array}{rl}
P(l, r)\ot P(l', r')
\cong &(\oplus_{i=0}^{l-1}2P(l+l'-1-2i, r+r'+i))\\
&\oplus(\oplus_{i=l'}^{l'+l-1}2P(n+l+l'-1-2i, r+r'+i))\\
&\oplus(\oplus_{c(l'+l-1)\<i\<l'-1}4P(n+l+l'-1-2i, r+r'+i))\\
&\oplus(\oplus_{1\<i\<c(n-l-l')}4P(l+l'-1+2i, r+r'-i)).\\
\end{array}$$
\end{proposition}

\begin{proof}
It is enough to show the proposition for $r=r'=0$.
By the discussion in \seref{2}, there is an exact sequence
$0\ra\O V(l', 0)\ra P(l',0)\ra V(l', 0)\ra 0$.
Applying $P(l, 0)\ot$ to the above sequences, one gets another sequence
$$0\ra P(l, 0)\ot\O V(l', 0)\ra P(l,0)\ot P(l', 0)\ra P(l, 0)\ot V(l', 0)\ra 0,$$
which is split since $P(l, 0)\ot V(l', 0)$ is projective.
Then the proposition follows from Proposition \ref{4.1} and its proof.
\end{proof}

\begin{corollary}\label{4.6}  Let $2\<l\<l'<n$ and $r, r'\in{\mathbb Z}_n$. Assume $t=l+l'-(n+1)\>0$. Then
$$\begin{array}{rl}
P(l, r)\ot P(l', r')
\cong&(\oplus_{i=c(t)}^{t}4P(l+l'-1-2i, r+r'+i))\\
&\oplus(\oplus_{i=t+1}^{l-1}2P(l+l'-1-2i, r+r'+i))\\
&\oplus(\oplus_{i=l'}^{n-1}2P(n+l+l'-1-2i, r+r'+i))\\
&\oplus(\oplus_{c(l'+l-1)\<i\<l'-1}4P(n+l+l'-1-2i, r+r'+i)).\\
\end{array}$$
\end{corollary}

\begin{proof}
It is similar to Proposition \ref{4.5}, by using Corollary \ref{4.2} and its proof.
\end{proof}

\begin{proposition}\label{4.7}  Let $1\<l, l'<n$, $r, r'\in{\mathbb Z}_n$, $\eta\in\mathbb{P}^1(k)$
and $s\>1$. Assume that $l+l'\<n$. Let $l_1={\rm min}\{l, l'\}$ and $l_2={\rm max}\{l, l'\}$. Then
$$\begin{array}{rl}
P(l, r)\ot M_s(l', r', \eta)
\cong &(\oplus_{i=0}^{l_1-1}sP(l+l'-1-2i, r+r'+i))\\
&\oplus(\oplus_{i=l_2}^{l+l'-1}sP(n+l+l'-1-2i, r+r'+i))\\
&\oplus(\oplus_{c(l+l'-1)\<i\<l_2-1}2sP(n+l+l'-1-2i, r+r'+i))\\
&\oplus(\oplus_{1\<i\<c(n-l-l')}2sP(l+l'-1+2i, r+r'-i)).\\
\end{array}$$
\end{proposition}

\begin{proof}
It is enough to show the proposition for $r=r'=0$.
By the structure of $M_s(l', 0, \eta)$, we have the following exact sequence
$$0\ra sP(l, 0)\ot V(l', 0)\ra P(l,0)\ot M_s(l', 0, \eta)\ra sP(l,0)\ot V(n-l', l')\ra 0,$$
which is split as pointed out before.
Then the proposition follows from an argument similar to the proof of Proposition \ref{4.1}.
\end{proof}

\begin{corollary}\label{4.8}  Let $2\<l, l'<n$, $r, r'\in{\mathbb Z}_n$, $\eta\in\mathbb{P}^1(k)$
and $s\>1$. Assume that $t=l+l'-(n+1)\>0$. Let $l_1={\rm min}\{l, l'\}$ and $l_2={\rm max}\{l, l'\}$. Then
$$\begin{array}{rl}
P(l, r)\ot M_s(l', r', \eta)
\cong&(\oplus_{i=c(t)}^{t}2sP(l+l'-1-2i, r+r'+i))\\
&\oplus(\oplus_{i=t+1}^{l_1-1}sP(l+l'-1-2i, r+r'+i))\\
&\oplus(\oplus_{i=l_2}^{n-1}sP(n+l+l'-1-2i, r+r'+i))\\
&\oplus(\oplus_{c(l+l'-1)\<i\<l_2-1}2sP(n+l+l'-1-2i, r+r'+i)).\\
\end{array}$$
\end{corollary}

\begin{proof}
It is similar to Proposition \ref{4.7}.
\end{proof}

\section{\bf Tensor product of two modules with Loewy length 2}\selabel{5}

In this section, we determine the tensor product of two non-simple non-projective indecomposable modules.
We will first consider the tensor product of two string modules.

\subsection{Tensor product of two string modules}\selabel{5.1}
~~

In this subsection, we determine the tensor product $\O^mV(l,r)\ot\O^sV(l',r)$ of two string modules,
where $m,s\in\mathbb Z$. By \cite[p.438]{EGST},
$\O^mV(l,r)\ot\O^sV(l',r)\cong\O^{m+s}(V(l,r)\ot V(l',0))\oplus P$
for some projective module $P$.  The first term on the right side of the isomorphism
is easily determined by Proposition \ref{3.1}.
However, the projective summand $P$ is not easy to determine in general.
For $m\>0$, we determine the tensor product by the induction on $m$ through the exact sequence
$$\begin{array}{rl}
0\ra\O^{m+1}V(l,r)\ot\O^sV(l',r')&\ra P(\O^{m}V(l,r))\ot\O^sV(l',r')\\
&\ra\O^{m}V(l,r)\ot\O^sV(l',r')\ra 0.\\
\end{array}$$
Here we use the following Lemma \ref{5.1}, which is obvious,
and the fact that $\O(M\oplus P)\cong\O M$ for any module
$M$ and projective module $P$.
Then applying the duality $(-)^*$, one achieves the
corresponding decompositions for $m<0$.

\begin{lemma}\label{5.1}
Let $0\ra N\ra P\ra M\ra 0$ be an exact sequence of modules over a finite dimensional algebra $A$,
where $P$ is projective. Then $P\cong P(M)\oplus Q$ for some submodule $Q$ of $P$.
Moreover, $Q$ is unique up to isomorphism, and $N\cong \O M\oplus Q$.
\end{lemma}

\begin{proposition}\label{5.2}
Let $1\<l\<l'<n$, $r, r'\in{\mathbb Z}_n$, $m\>0$ and $s\>1$. Assume that $l+l'\<n$, and
let $m_1={\rm min}\{m, s\}$ and $m_2={\rm max}\{m, s\}$. Let $P$ be the module
$$\begin{array}{l}
(\oplus_{c(l'+l-1)\<i\<l'-1}(m+\frac{1-(-1)^m}{2})(s+\frac{1+(-1)^s}{2})P(n+l+l'-1-2i, r+r'+i))\\
\oplus(\oplus_{1\<i\<c(n-l-l')}(m+\frac{1-(-1)^m}{2})(s+\frac{1-(-1)^s}{2})P(l+l'-1+2i, r+r'-i)).\\
\end{array}$$
$(1)$ If $m+s$ is even, then
$$\begin{array}{rl}
\O^{m}V(l, r)\ot\O^sV(l', r')
\cong&(\oplus_{i=0}^{l-1}\O^{m+s}V(l+l'-1-2i, r+r'+i))\\
&\oplus(\oplus_{i=0}^{l-1}msP(l+l'-1-2i, r+r'+i))\oplus P\\
\end{array}$$
and
$$\begin{array}{rl}
&\O^mV(l, r)\ot\O^{-s}V(l', r')\\
\cong&(\oplus_{i=0}^{l-1}\O^{m-s}V(l+l'-1-2i, r+r'+i))\\
&\oplus(\oplus_{i=l'}^{l+l'-1}m_1(m_2+1)P(n+l+l'-1-2i, r+r'+i))\oplus P.\\
\end{array}$$
$(2)$ If $m+s$ is odd, then
$$\begin{array}{rl}
\O^mV(l, r)\ot\O^sV(l', r')
\cong&(\oplus_{i=0}^{l-1}\O^{m+s}V(l+l'-1-2i, r+r'+i))\\
&\oplus(\oplus_{i=l'}^{l+l'-1}msP(n+l+l'-1-2i, r+r'+i))\oplus P\\
\end{array}$$
and
$$\begin{array}{rl}
\O^mV(l, r)\ot\O^{-s}V(l', r')
\cong&(\oplus_{i=0}^{l-1}\O^{m-s}V(l+l'-1-2i, r+r'+i))\\
&\oplus(\oplus_{i=0}^{l-1}m_1(m_2+1)P(l+l'-1-2i, r+r'+i))\oplus P.\\
\end{array}$$
\end{proposition}

\begin{proof}
It is enough to show the proposition
for $r=r'=0$. We prove it by the induction on $m$.
For $m=0$, it follows from Proposition \ref{3.6}.
Now let $m>0$.
We only consider the case that $m$ and $s$ are both even since the proofs are
similar for the other cases. In this case, we have an exact sequence
$$\begin{array}{rl}
0\ra\O^{m}V(l,0)\ot\O^{\pm s}V(l', 0)&\ra mP(n-l, l)\ot\O^{\pm s}V(l',0)\\
&\ra\O^{m-1}V(l, 0)\ot\O^{\pm s}V(l',0)\ra 0.\\
\end{array}$$
From $1\<l\<l'<n$ and $l+l'\<n$, one gets that $1\<l'\<n-l<n$ and $n-l+l'\>n$.
Moreover, $n-l+l'-(n+1)=l'-l-1$. Hence by Corollary \ref{4.4} together
with Proposition \ref{4.1} for $l+l'=n$ and $l=l'$,
Corollary \ref{4.2} for $l+l'=n$ and $l<l'$, and Corollary \ref{4.3} for $l+l'<n$ and $l=l'$, we have
$$\begin{array}{rl}
&mP(n-l, l)\ot\O^{\pm s}V(l', 0)\\
\cong&(\oplus_{c(l'-l-1)\<i\<l'-l-1}2m(s+1)P(n-l+l'-1-2i, l+i))\\
&\oplus(\oplus_{i=l'-l}^{l'-1}m(s+1)P(n-l+l'-1-2i, l+i))\\
&\oplus(\oplus_{i=n-l}^{n-1}msP(2n-l+l'-1-2i, l+i))\\
&\oplus(\oplus_{c(n-l+l'-1)\<i\<n-l-1}2msP(2n-l+l'-1-2i, l+i))\\
\cong&(\oplus_{c(l+l'-1)\<i\<l'-1}2m(s+1)P(n+l+l'-1-2i, i))\\
&\oplus(\oplus_{i=l'}^{l+l'-1}m(s+1)P(n+l+l'-1-2i, i))\\
&\oplus(\oplus_{i=0}^{l-1}msP(l+l'-1-2i, i))\\
&\oplus(\oplus_{1\<i\<c(n-l-l')}2msP(l+l'-1+2i, -i)).\\
\end{array}$$
Note that $m-1+s$ and $m-1$ are both odd.
By the induction hypothesis, we have
$$\begin{array}{rl}
&\O^{m-1}V(l, 0)\ot\O^sV(l', 0)\\
\cong&(\oplus_{i=0}^{l-1}\O^{m-1+s}V(l+l'-1-2i, i))\\
&\oplus(\oplus_{i=l'}^{l+l'-1}(m-1)sP(n+l+l'-1-2i, i))\\
&\oplus(\oplus_{c(l'+l-1)\<i\<l'-1}m(s+1)P(n+l+l'-1-2i, i))\\
&\oplus(\oplus_{1\<i\<c(n-l-l')}msP(l+l'-1+2i, -i)).\\
\end{array}$$
It is easy to check that $\oplus_{i=l'}^{l+l'-1}(m+s)P(n+l+l'-1-2i, i)$ is a projective cover of
$\oplus_{i=0}^{l-1}\O^{m-1+s}V(l+l'-1-2i, i)$.
Hence we have
$$\begin{array}{rl}
&mP(n-l, l)\ot\O^{s}V(l', 0)\\
\cong&P(\O^{m-1}V(l, 0)\ot\O^sV(l', 0))\\
&\oplus(\oplus_{c(l+l'-1)\<i\<l'-1}m(s+1)P(n+l+l'-1-2i, i))\\
&\oplus(\oplus_{i=0}^{l-1}msP(l+l'-1-2i, i))\\
&\oplus(\oplus_{1\<i\<c(n-l-l')}msP(l+l'-1+2i, -i)).\\
\end{array}$$
It follows from Lemma \ref{5.1} that
$$\begin{array}{rl}
\O^mV(l,0)\ot\O^sV(l', 0)
\cong&(\oplus_{i=0}^{l-1}\O^{m+s}V(l+l'-1-2i, i))\\
&\oplus(\oplus_{i=0}^{l-1}msP(l+l'-1-2i, i))\\
&\oplus(\oplus_{c(l+l'-1)\<i\<l'-1}m(s+1)P(n+l+l'-1-2i, i))\\
&\oplus(\oplus_{1\<i\<c(n-l-l')}msP(l+l'-1+2i, -i)).\\
\end{array}$$
If $s\>m$, then $s>m-1$. Hence by the induction hypothesis, we have
$$\begin{array}{rl}
&\O^{m-1}V(l, 0)\ot\O^{-s}V(l', 0)\\
\cong&(\oplus_{i=0}^{l-1}\O^{m-1-s}V(l+l'-1-2i, i))\\
&\oplus(\oplus_{i=0}^{l-1}(m-1)(s+1)P(l+l'-1-2i, i))\\
&\oplus(\oplus_{c(l+l'-1)\<i\<l'-1}m(s+1)P(n+l+l'-1-2i, i))\\
&\oplus(\oplus_{1\<i\<c(n-l-l')}msP(l+l'-1+2i, -i)).\\
\end{array}$$
In this case, $s-m+1\>1$ is odd. Hence
$\oplus_{i=0}^{l-1}(s-m+1)P(l+l'-1-2i, i)$ is a projective cover of
$\oplus_{i=0}^{l-1}\O^{m-1-s}V(l+l'-1-2i, i)$. Thus, a similar argument as above shows that
$$\begin{array}{rl}
\O^{m}V(l, 0)\ot\O^{-s}V(l', 0)
\cong&(\oplus_{i=0}^{l-1}\O^{m-s}V(l+l'-1-2i, i))\\
&\oplus(\oplus_{i=l'}^{l+l'-1}m(s+1)P(n+l+l'-1-2i, i))\\
&\oplus(\oplus_{c(l+l'-1)\<i\<l'-1}m(s+1)P(n+l+l'-1-2i, i))\\
&\oplus(\oplus_{1\<i\<c(n-l-l')}msP(l+l'-1+2i, -i)).\\
\end{array}$$
If $m>s$, then $m-1\>s$. Hence by the induction hypothesis, we have
$$\begin{array}{rl}
&\O^{m-1}V(l, 0)\ot\O^{-s}V(l', 0)\\
\cong&(\oplus_{i=0}^{l-1}\O^{m-1-s}V(l+l'-1-2i, i))\\
&\oplus(\oplus_{i=0}^{l-1}smP(l+l'-1-2i, i))\\
&\oplus(\oplus_{c(l+l'-1)\<i\<l'-1}m(s+1)P(n+l+l'-1-2i, i))\\
&\oplus(\oplus_{1\<i\<c(n-l-l')}msP(l+l'-1+2i, -i)).\\
\end{array}$$
In this case, $m-1-s\>0$ is odd. Hence $\oplus_{i=l'}^{l+l'-1}(m-s)P(n+l+l'-1-2i, i)$ is a projective cover of
$\oplus_{i=0}^{l-1}\O^{m-1-s}V(l+l'-1-2i, i)$ as above, and so similarly,
$$\begin{array}{rl}
\O^{m}V(l, 0)\ot\O^{-s}V(l', 0)
\cong&(\oplus_{i=0}^{l-1}\O^{m-s}V(l+l'-1-2i, i))\\
&\oplus(\oplus_{i=l'}^{l+l'-1}s(m+1)P(n+l+l'-1-2i, i))\\
&\oplus(\oplus_{c(l+l'-1)\<i\<l'-1}m(s+1)P(n+l+l'-1-2i, i))\\
&\oplus(\oplus_{1\<i\<c(n-l-l')}msP(l+l'-1+2i, -i)),\\
\end{array}$$
as desired. This completes the proof.
\end{proof}

\begin{corollary}\label{5.3}
Let $1\<l\<l'<n$, $r, r'\in{\mathbb Z}_n$ and $s, m\>1$. Assume that $l+l'\<n$, and
let $m_1={\rm min}\{m, s\}$ and $m_2={\rm max}\{m, s\}$. Let $P$ be the module
$$\begin{array}{l}
(\oplus_{c(l'+l-1)\<i\<l'-1}(m+\frac{1-(-1)^m}{2})(s+\frac{1+(-1)^s}{2})P(n+l+l'-1-2i, r+r'+i))\\
\oplus(\oplus_{1\<i\<c(n-l-l')}(m+\frac{1-(-1)^m}{2})(s+\frac{1-(-1)^s}{2})P(l+l'-1+2i, r+r'-i)).\\
\end{array}$$
$(1)$ If $m+s$ is even, then
$$\begin{array}{rl}
\O^{-m}V(l, r)\ot\O^{-s}V(l', r')
\cong&(\oplus_{i=0}^{l-1}\O^{-(m+s)}V(l+l'-1-2i, r+r'+i))\\
&\oplus(\oplus_{i=0}^{l-1}msP(l+l'-1-2i, r+r'+i))\oplus P\\
\end{array}$$
and
$$\begin{array}{rl}
&\O^{-m}V(l, r)\ot\O^{s}V(l', r')\\
\cong&(\oplus_{i=0}^{l-1}\O^{s-m}V(l+l'-1-2i, r+r'+i))\\
&\oplus(\oplus_{i=l'}^{l+l'-1}m_1(m_2+1)P(n+l+l'-1-2i, r+r'+i))\oplus P.\\
\end{array}$$
$(2)$ If $m+s$ is odd, then
$$\begin{array}{rl}
\O^{-m}V(l, r)\ot\O^{-s}V(l', r')
\cong&(\oplus_{i=0}^{l-1}\O^{-(m+s)}V(l+l'-1-2i, r+r'+i))\\
&\oplus(\oplus_{i=l'}^{l+l'-1}msP(n+l+l'-1-2i, r+r'+i))\oplus P\\
\end{array}$$
and
$$\begin{array}{rl}
\O^{-m}V(l, r)\ot\O^{s}V(l', r')
\cong&(\oplus_{i=0}^{l-1}\O^{s-m}V(l+l'-1-2i, r+r'+i))\\
&\oplus(\oplus_{i=0}^{l-1}m_1(m_2+1)P(l+l'-1-2i, r+r'+i))\oplus P.\\
\end{array}$$
\end{corollary}

\begin{proof}
Applying the duality $(-)^*$ to the isomorphisms in Proposition \ref{5.2},
the corollary follows from Lemma \ref{3.12}.
\end{proof}

\begin{proposition}\label{5.4}
Let $1\<l\<l'<n$, $r, r'\in{\mathbb Z}_n$, $m\>0$ and $s\>1$. Assume that
$t=l+l'-(n+1)\>0$, and let $m_1={\rm min}\{m, s\}$ and $m_2={\rm max}\{m, s\}$. Let
$$\begin{array}{l}
\ \ \ \ P=(\oplus_{i=c(t)}^{t}(m+\frac{1+(-1)^m}{2})(s+\frac{1+(-1)^s}{2})P(l+l'-1-2i, r+r'+i))\\
\oplus(\oplus_{c(l+l'-1)\<i\<l'-1}(m+\frac{1-(-1)^m}{2})(s+\frac{1+(-1)^s}{2})P(n+l+l'-1-2i, r+r'+i)).\\
\end{array}$$
$(1)$ If $m+s$ is even, then
$$\begin{array}{rl}
\O^mV(l, r)\ot\O^sV(l', r')
\cong&(\oplus_{i=t+1}^{l-1}\O^{m+s}V(l+l'-1-2i, r+r'+i))\\
&\oplus(\oplus_{i=t+1}^{l-1}msP(l+l'-1-2i, r+r'+i))\oplus P\\
\end{array}$$
and
$$\begin{array}{rl}
&\O^{m}V(l, r)\ot\O^{-s}V(l', r')\\
\cong&(\oplus_{i=t+1}^{l-1}\O^{m-s}V(l+l'-1-2i, r+r'+i))\\
&\oplus(\oplus_{i=l'}^{n-1}m_1(m_2+1)P(n+l+l'-1-2i, r+r'+i))\oplus P.\\
\end{array}$$
$(2)$ If $m+s$ is odd, then
$$\begin{array}{rl}
\O^mV(l, r)\ot\O^sV(l', r')
\cong&(\oplus_{i=t+1}^{l-1}\O^{m+s}V(l+l'-1-2i, r+r'+i))\\
&\oplus(\oplus_{i=l'}^{n-1}msP(n+l+l'-1-2i, r+r'+i))\oplus P\\
\end{array}$$
and
$$\begin{array}{rl}
&\O^{m}V(l, r)\ot\O^{-s}V(l', r')\\
\cong&(\oplus_{i=t+1}^{l-1}\O^{m-s}V(l+l'-1-2i, r+r'+i))\\
&\oplus(\oplus_{i=t+1}^{l-1}m_1(m_2+1)P(l+l'-1-2i, r+r'+i))\oplus P.\\
\end{array}$$
\end{proposition}

\begin{proof}
It is similar to Proposition \ref{5.2}, where we use Proposition \ref{3.7} for $m=0$.
\end{proof}

\begin{corollary}\label{5.5}
Let $1\<l\<l'<n$, $r, r'\in{\mathbb Z}_n$ and $s, m\>1$. Assume that $t=l+l'-(n+1)\>0$, and
let $m_1={\rm min}\{m, s\}$ and $m_2={\rm max}\{m, s\}$. Let
$$\begin{array}{l}
\ \ \ \ P=(\oplus_{i=c(t)}^{t}(m+\frac{1+(-1)^m}{2})(s+\frac{1+(-1)^s}{2})P(l+l'-1-2i, r+r'+i))\\
\oplus(\oplus_{c(l+l'-1)\<i\<l'-1}(m+\frac{1-(-1)^m}{2})(s+\frac{1+(-1)^s}{2})P(n+l+l'-1-2i, r+r'+i)).\\
\end{array}$$
$(1)$ If $m+s$ is even, then
$$\begin{array}{rl}
\O^{-m}V(l, r)\ot\O^{-s}V(l', r')
\cong&(\oplus_{i=t+1}^{l-1}\O^{-(m+s)}V(l+l'-1-2i, r+r'+i))\\
&\oplus(\oplus_{i=t+1}^{l-1}msP(l+l'-1-2i, r+r'+i))\oplus P\\
\end{array}$$
and
$$\begin{array}{rl}
&\O^{-m}V(l, r)\ot\O^{s}V(l', r')\\
\cong&(\oplus_{i=t+1}^{l-1}\O^{s-m}V(l+l'-1-2i, r+r'+i))\\
&\oplus(\oplus_{i=l'}^{n-1}m_1(m_2+1)P(n+l+l'-1-2i, r+r'+i))\oplus P.\\
\end{array}$$
$(3)$ If $m+s$ is odd, then
$$\begin{array}{rl}
\O^{-m}V(l, r)\ot\O^{-s}V(l', r')
\cong&(\oplus_{i=t+1}^{l-1}\O^{-(s+m)}V(l+l'-1-2i, r+r'+i))\\
&\oplus(\oplus_{i=l'}^{n-1}msP(n+l+l'-1-2i, r+r'+i))\oplus P\\
\end{array}$$
and
$$\begin{array}{rl}
&\O^{-m}V(l, r)\ot\O^{s}V(l', r')\\
\cong&(\oplus_{i=t+1}^{l-1}\O^{s-m}V(l+l'-1-2i, r+r'+i))\\
&\oplus(\oplus_{i=t+1}^{l-1}m_1(m_2+1)P(l+l'-1-2i, r+r'+i))\oplus P.\\
\end{array}$$
\end{corollary}

\begin{proof}
Applying the duality $(-)^*$ to the isomorphisms in Proposition \ref{5.4}, the corollary follows from Lemmas \ref{3.12}.
\end{proof}

\subsection{Tensor product of a string module with a band module}\selabel{5.2}
~~

In this subsection, we determine the tensor product $M=\O^mV(l,r)\ot M_s(l', r', \eta)$
of a string module with a band module. Using the same method as in the last subsection
(replacing $\O^sV(l',r')$ by $M_s(l',r',\eta)$ there), we can determine $M$.

\begin{proposition}\label{5.6}
Let $1\<l\<l'<n$, $r, r'\in{\mathbb Z}_n$, $\eta\in\mathbb{P}^1(k)$, $s\>1$ and $m\>0$. Assume that $l+l'\<n$.\\
$(1)$ If $m$ is odd, then
$$\begin{array}{rl}
&\O^mV(l, r)\ot M_s(l', r', \eta)\\
\cong&(\oplus_{i=l'}^{l+l'-1}M_s(n+l+l'-1-2i, r+r'+i, -\eta q^{l'}\frac{(2i-l-l'+1)_q}{(l')_q}))\\
&\oplus(\oplus_{i=0}^{l-1}msP(l+l'-1-2i, r+r'+i))\\
&\oplus(\oplus_{c(l+l'-1)\<i\<l'-1}(m+1)sP(n+l+l'-1-2i, r+r'+i))\\
&\oplus(\oplus_{1\<i\<c(n-l-l')}(m+1)sP(l+l'-1+2i, r+r'-i)).\\
\end{array}$$
$(2)$ If $m$ is even, then
$$\begin{array}{rl}
&\O^mV(l, r)\ot M_s(l', r', \eta)\\
\cong&(\oplus_{i=0}^{l-1}M_s(l+l'-1-2i, r+r'+i, \eta q^{2i-l+1}\frac{(l+l'-1-2i)_q}{(l')_q}))\\
&\oplus(\oplus_{i=l'}^{l+l'-1}msP(n+l+l'-1-2i, r+r'+i))\\
&\oplus(\oplus_{c(l+l'-1)\<i\<l'-1}msP(n+l+l'-1-2i, r+r'+i))\\
&\oplus(\oplus_{1\<i\<c(n-l-l')}msP(l+l'-1+2i, r+r'-i)).\\
\end{array}$$
\end{proposition}

\begin{proof}
It is enough to show the proposition for $r=r'=0$.
We prove it by the induction on $m$.
For $m=0$, it follows from Theorem \ref{3.16}. Now let $m>0$.

We only consider the case that $m$ is odd since the proof is similar when $m$ is even.
In this case, $m-1$ is even, and hence there is an exact sequence
$$\begin{array}{rl}
0\ra \O^mV(l, 0)\ot M_s(l',0,\eta)&\ra mP(l, 0)\ot M_s(l',0,\eta)\\
&\ra \O^{m-1}V(l, 0)\ot M_s(l',0,\eta)\ra 0.\\
\end{array}$$
By the induction hypothesis, we have
$$\begin{array}{rl}
&\O^{m-1}V(l, 0)\ot M_s(l', 0, \eta)\\
\cong&(\oplus_{i=0}^{l-1}M_s(l+l'-1-2i, i, \eta q^{2i-l+1}\frac{(l+l'-1-2i)_q}{(l')_q}))\\
&\oplus(\oplus_{i=l'}^{l+l'-1}(m-1)sP(n+l+l'-1-2i, i))\\
&\oplus(\oplus_{c(l+l'-1)\<i\<l'-1}(m-1)sP(n+l+l'-1-2i, i))\\
&\oplus(\oplus_{1\<i\<c(n-l-l')}(m-1)sP(l+l'-1+2i, -i)).\\
\end{array}$$
Note that $\oplus_{i=l'}^{l+l'-1}sP(n+l+l'-1-2i, i)$ is a projective cover
of $\oplus_{i=0}^{l-1}M_s(l+l'-1-2i,i,\eta q^{2i-l+1}\frac{(l+l'-1-2i)_q}{(l')_q})$
and
$$\begin{array}{rl}
&\O(\oplus_{i=0}^{l-1}M_s(l+l'-1-2i,i,\eta q^{2i-l+1}\frac{(l+l'-1-2i)_q}{(l')_q}))\\
\cong&\oplus_{i=l'}^{l+l'-1}M_s(n+l+l'-1-2i, i, -\eta q^{l'}\frac{(2i-l-l'+1)_q}{(l')_q}).\\
\end{array}$$
Hence by Proposition \ref{4.7}, we have
$$\begin{array}{rl}
mP(l, 0)\ot M_s(l', 0, \eta)
\cong &(\oplus_{i=0}^{l-1}msP(l+l'-1-2i, i))\\
&\oplus(\oplus_{i=l'}^{l+l'-1}msP(n+l+l'-1-2i, i))\\
&\oplus(\oplus_{c(l+l'-1)\<i\<l'-1}2msP(n+l+l'-1-2i, i))\\
&\oplus(\oplus_{1\<i\<c(n-l-l')}2msP(l+l'-1+2i, -i))\\
\cong&P(\O^{m-1}V(l, 0)\ot M_s(l', 0, \eta))\\
&\oplus(\oplus_{i=0}^{l-1}msP(l+l'-1-2i, i))\\
&\oplus(\oplus_{c(l+l'-1)\<i\<l'-1}(m+1)sP(n+l+l'-1-2i, i))\\
&\oplus(\oplus_{1\<i\<c(n-l-l')}(m+1)sP(l+l'-1+2i, -i)).\\
\end{array}$$
Then it follows from Lemma \ref{5.1} that
$$\begin{array}{rl}
\O^mV(l, 0)\ot M_s(l', 0, \eta)
\cong&(\oplus_{i=l'}^{l+l'-1}M_s(n+l+l'-1-2i, i, -\eta q^{l'}\frac{(2i-l-l'+1)_q}{(l')_q}))\\
&\oplus(\oplus_{i=0}^{l-1}msP(l+l'-1-2i, i))\\
&\oplus(\oplus_{c(l+l'-1)\<i\<l'-1}(m+1)sP(n+l+l'-1-2i, i))\\
&\oplus(\oplus_{1\<i\<c(n-l-l')}(m+1)sP(l+l'-1+2i, -i)),\\
\end{array}$$
as desired. This completes the proof.
\end{proof}

\begin{corollary}\label{5.7}
Let $1\<l\<l'<n$, $r, r'\in{\mathbb Z}_n$, $\eta\in\mathbb{P}^1(k)$ and $s, m\>1$. Assume that $l+l'\<n$.\\
$(1)$ If $m$ is odd, then
$$\begin{array}{rl}
&\O^{-m}V(l, r)\ot M_s(l', r', \eta)\\
\cong&(\oplus_{i=l'}^{l+l'-1}M_s(n+l+l'-1-2i, r+r'+i, -\eta q^{l'}\frac{(2i-l-l'+1)_q}{(l')_q}))\\
&\oplus(\oplus_{i=l'}^{l+l'-1}msP(n+l+l'-1-2i, r+r'+i))\\
&\oplus(\oplus_{c(l+l'-1)\<i\<l'-1}(m+1)sP(n+l+l'-1-2i, r+r'+i))\\
&\oplus(\oplus_{1\<i\<c(n-l-l')}(m+1)sP(l+l'-1+2i, r+r'-i)).\\
\end{array}$$
$(2)$ If $m$ is even, then
$$\begin{array}{rl}
&\O^{-m}V(l, r)\ot M_s(l', r', \eta)\\
\cong&(\oplus_{i=0}^{l-1}M_s(l+l'-1-2i, r+r'+i, \eta q^{2i-l+1}\frac{(l+l'-1-2i)_q}{(l')_q}))\\
&\oplus(\oplus_{i=0}^{l-1}msP(l+l'-1-2i, r+r'+i))\\
&\oplus(\oplus_{c(l+l'-1)\<i\<l'-1}msP(n+l+l'-1-2i, r+r'+i))\\
&\oplus(\oplus_{1\<i\<c(n-l-l')}msP(l+l'-1+2i, r+r'-i)).\\
\end{array}$$
\end{corollary}

\begin{proof}
It is enough to show the corollary for $r=r'=0$.
Since $1\<l\<l'<n$ and $l+l'\<n$, we have $1\<l\<n-l'<n$ and $l+n-l'\<n$.

(1) Assume that $m$ is odd. Then by Proposition \ref{5.6}, we have
$$\begin{array}{rl}
&\O^mV(l, 1-l)\ot M_s(n-l', 1, -\eta q^{l'})\\
\cong&(\oplus_{i=n-l'}^{l+n-l'-1}M_s(2n+l-l'-1-2i, 2-l+i, \eta\frac{(2i-l-n+l'+1)_q}{(n-l')_q}))\\
&\oplus(\oplus_{i=0}^{l-1}msP(l+n-l'-1-2i, 2-l+i))\\
&\oplus(\oplus_{c(l+n-l'-1)\<i\<n-l'-1}(m+1)sP(2n+l-l'-1-2i, 2-l+i))\\
&\oplus(\oplus_{1\<i\<c(l'-l)}(m+1)sP(l+n-l'-1+2i, 2-l-i)).\\
\end{array}$$
Then applying the duality $(-)^*$ to the above isomorphism and using Lemmas \ref{3.12} and \ref{3.13},
a tedious but standard computation shows that
$$\begin{array}{rl}
&\O^{-m}V(l, 0)\ot M_s(l', 0, \eta)\\
\cong&(\oplus_{i=l'}^{l+l'-1}M_s(n+l+l'-1-2i, i, -\eta q^{l'}\frac{(2i-l-l'+1)_q}{(l')_q}))\\
&\oplus(\oplus_{i=l'}^{l+l'-1}msP(n+l+l'-1-2i, i))\\
&\oplus(\oplus_{c(l+l'-1)\<i\<l'-1}(m+1)sP(n+l+l'-1-2i, i))\\
&\oplus(\oplus_{1\<i\<c(n-l-l')}(m+1)sP(l+l'-1+2i, -i)).\\
\end{array}$$

(2) It is similar to Part (1).
\end{proof}

\begin{proposition}\label{5.8}
Let $2\<l\<l'<n$, $r, r'\in{\mathbb Z}_n$, $\eta\in\mathbb{P}^1(k)$, $s\>1$ and $m\>0$.
Assume that $t=l+l'-(n+1)\>0$.\\
$(1)$ If $m$ is odd, then
$$\begin{array}{rl}
&\O^mV(l, r)\ot M_s(l', r', \eta)\\
\cong&(\oplus_{i=l'}^{n-1}M_s(n+l+l'-1-2i, r+r'+i, -\eta q^{l'}\frac{(2i-l-l'+1)_q}{(l')_q}))\\
&\oplus(\oplus_{i=t+1}^{l-1}msP(l+l'-1-2i, r+r'+i))\\
&\oplus(\oplus_{i=c(t)}^tmsP(l+l'-1-2i, r+r'+i))\\
&\oplus(\oplus_{c(l+l'-1)\<i\<l'-1}(m+1)sP(n+l+l'-1-2i, r+r'+i))\\
\end{array}$$
$(2)$ If $m$ is even, then
$$\begin{array}{rl}
&\O^mV(l, r)\ot M_s(l', r', \eta)\\
\cong&(\oplus_{i=t+1}^{l-1}M_s(l+l'-1-2i, r+r'+i, \eta q^{2i-l+1}\frac{(l+l'-1-2i)_q}{(l')_q}))\\
&\oplus(\oplus_{i=l'}^{n-1}msP(n+l+l'-1-2i, r+r'+i))\\
&\oplus(\oplus_{i=c(t)}^t(m+1)sP(l+l'-1-2i, r+r'+i))\\
&\oplus(\oplus_{c(l+l'-1)\<i\<l'-1}msP(n+l+l'-1-2i, r+r'+i)).\\
\end{array}$$
\end{proposition}

\begin{proof}
It is similar to Proposition \ref{5.6}, where we use Theorem \ref{3.17} for $m=0$.
\end{proof}

\begin{corollary}\label{5.9}
Let $1\<l'<l<n$, $r, r'\in{\mathbb Z}_n$, $\eta\in\mathbb{P}^1(k)$ and $s, m\>1$. Assume that $l+l'\<n$.\\
$(1)$ If $m$ is odd, then
$$\begin{array}{rl}
&\O^{-m}V(l, r)\ot M_s(l', r', \eta)\\
\cong&(\oplus_{i=l}^{l+l'-1}M_s(n+l+l'-1-2i, r+r'+i, -\eta q^{l'}\frac{(2i-l-l'+1)_q}{(l')_q}))\\
&\oplus(\oplus_{i=l}^{l+l'-1}msP(n+l+l'-1-2i, r+r'+i))\\
&\oplus(\oplus_{i=c(l+l'-1)}^{l-1}msP(n+l+l'-1-2i, r+r'+i))\\
&\oplus(\oplus_{1\<i\<c(n-l-l')}(m+1)sP(l+l'-1+2i, r+r'-i)).\\
\end{array}$$
$(2)$ If $m$ is even, then
$$\begin{array}{rl}
&\O^{-m}V(l, r)\ot M_s(l', r', \eta)\\
\cong&(\oplus_{i=0}^{l'-1}M_s(l+l'-1-2i, r+r'+i, \eta q^{2i-l+1}\frac{(l+l'-1-2i)_q}{(l')_q}))\\
&\oplus(\oplus_{i=0}^{l'-1}msP(l+l'-1-2i, r+r'+i))\\
&\oplus(\oplus_{i=c(l+l'-1)}^{l-1}(m+1)sP(n+l+l'-1-2i, r+r'+i))\\
&\oplus(\oplus_{1\<i\<c(n-l-l')}msP(l+l'-1+2i, r+r'-i)).\\
\end{array}$$
\end{corollary}

\begin{proof}
It is similar to Corollary \ref{5.7} by using the duality $(-)^*$,
Lemmas \ref{3.12}-\ref{3.13}, and Proposition \ref{5.8}.
\end{proof}

\begin{proposition}\label{5.10}
Let $1\<l'<l<n$, $r, r'\in{\mathbb Z}_n$, $\eta\in\mathbb{P}^1(k)$, $s\>1$ and $m\>0$. Assume that $l+l'\<n$.\\
$(1)$ If $m$ is odd, then
$$\begin{array}{rl}
&\O^mV(l, r)\ot M_s(l', r', \eta)\\
\cong&(\oplus_{i=l}^{l+l'-1}M_s(n+l+l'-1-2i, r+r'+i, -\eta q^{l'}\frac{(2i-l-l'+1)_q}{(l')_q}))\\
&\oplus(\oplus_{i=0}^{l'-1}msP(l+l'-1-2i, r+r'+i))\\
&\oplus(\oplus_{i=c(l+l'-1)}^{l-1}msP(n+l+l'-1-2i, r+r'+i))\\
&\oplus(\oplus_{1\<i\<c(n-l-l')}(m+1)sP(l+l'-1+2i, r+r'-i)).\\
\end{array}$$
$(2)$ If $m$ is even, then
$$\begin{array}{rl}
&\O^mV(l, r)\ot M_s(l', r', \eta)\\
\cong&(\oplus_{i=0}^{l'-1}M_s(l+l'-1-2i, r+r'+i, \eta q^{2i-l+1}\frac{(l+l'-1-2i)_q}{(l')_q}))\\
&\oplus(\oplus_{i=l}^{l+l'-1}msP(n+l+l'-1-2i, r+r'+i))\\
&\oplus(\oplus_{i=c(l+l'-1)}^{l-1}(m+1)sP(n+l+l'-1-2i, r+r'+i))\\
&\oplus(\oplus_{1\<i\<c(n-l-l')}msP(l+l'-1+2i, r+r'-i)).\\
\end{array}$$
\end{proposition}

\begin{proof}
It is similar to Proposition \ref{5.6}, where we use Theorem \ref{3.16} for $m=0$.
\end{proof}

\begin{corollary}\label{5.11}
Let $2\<l\<l'<n$, $r, r'\in{\mathbb Z}_n$, $\eta\in\mathbb{P}^1(k)$ and $s, m\>1$. Assume that
$t=l+l'-(n+1)\>0$.\\
$(1)$ If $m$ is odd, then
$$\begin{array}{rl}
&\O^{-m}V(l, r)\ot M_s(l', r', \eta)\\
\cong&(\oplus_{i=l'}^{n-1}M_s(n+l+l'-1-2i, r+r'+i, -\eta q^{l'}\frac{(2i-l-l'+1)_q}{(l')_q}))\\
&\oplus(\oplus_{i=l'}^{n-1}msP(n+l+l'-1-2i, r+r'+i))\\
&\oplus(\oplus_{i=c(t)}^{t}msP(l+l'-1-2i, r+r'+i))\\
&\oplus(\oplus_{c(l+l'-1)\<i\<l'-1}(m+1)sP(n+l+l'-1-2i, r+r'+i)).\\
\end{array}$$
$(2)$ If $m$ is even, then
$$\begin{array}{rl}
&\O^{-m}V(l, r)\ot M_s(l', r', \eta)\\
\cong&(\oplus_{i=t+1}^{l-1}M_s(l+l'-1-2i, r+r'+i, \eta q^{2i-l+1}\frac{(l+l'-1-2i)_q}{(l')_q}))\\
&\oplus(\oplus_{i=t+1}^{l-1}msP(l+l'-1-2i, r+r'+i))\\
&\oplus(\oplus_{i=c(t)}^{t}(m+1)sP(l+l'-1-2i, r+r'+i))\\
&\oplus(\oplus_{c(l+l'-1)\<i\<l'-1}msP(n+l+l'-1-2i, r+r'+i)).\\
\end{array}$$
\end{corollary}

\begin{proof}
It is similar to Corollary \ref{5.7} by using the duality $(-)^*$,
Lemmas \ref{3.12}-\ref{3.13}, and Proposition \ref{5.10}.
\end{proof}

\begin{proposition}\label{5.12}
Let $2\<l'<l<n$, $r, r'\in{\mathbb Z}_n$, $\eta\in\mathbb{P}^1(k)$, $s\>1$ and $m\>0$.
Assume that $t=l+l'-(n+1)\>0$.\\
$(1)$ If $m$ is odd, then
$$\begin{array}{rl}
&\O^m V(l, r)\ot M_s(l', r', \eta)\\
\cong&(\oplus_{i=l}^{n-1}M_s(n+l+l'-1-2i, r+r'+i, -\eta q^{l'}\frac{(2i-l-l'+1)_q}{(l')_q}))\\
&\oplus(\oplus_{i=t+1}^{l'-1}msP(l+l'-1-2i, r+r'+i))\\
&\oplus(\oplus_{i=c(t)}^tmsP(l+l'-1-2i, r+r'+i))\\
&\oplus(\oplus_{i=c(l+l'-1)}^{l-1}msP(n+l+l'-1-2i, r+r'+i)).\\
\end{array}$$
$(2)$ If $m$ is even, then
$$\begin{array}{rl}
&\O^mV(l, r)\ot M_s(l', r', \eta)\\
\cong&(\oplus_{i=t+1}^{l'-1}M_s(l+l'-1-2i, r+r'+i, \eta q^{2i-l+1}\frac{(l+l'-1-2i)_q}{(l')_q}))\\
&\oplus(\oplus_{i=l}^{n-1}msP(n+l+l'-1-2i, r+r'+i))\\
&\oplus(\oplus_{i=c(t)}^t(m+1)sP(l+l'-1-2i, r+r'+i))\\
&\oplus(\oplus_{i=c(l+l'-1)}^{l-1}(m+1)sP(n+l+l'-1-2i, r+r'+i)).\\
\end{array}$$
\end{proposition}

\begin{proof}
It is similar to Proposition \ref{5.6}, where we use Theorem \ref{3.17} for $m=0$.
\end{proof}

\begin{corollary}\label{5.13}
Let $2\<l'<l<n$, $r, r'\in{\mathbb Z}_n$, $\eta\in\mathbb{P}^1(k)$ and $s, m\>1$. Assume that
$t=l+l'-(n+1)\>0$.\\
$(1)$ If $m$ is odd, then
$$\begin{array}{rl}
&\O^{-m}V(l, r)\ot M_s(l', r', \eta)\\
\cong&(\oplus_{i=l}^{n-1}M_s(n+l+l'-1-2i, r+r'+i, -\eta q^{l'}\frac{(2i-l-l'+1)_q}{(l')_q}))\\
&\oplus(\oplus_{i=l}^{n-1}msP(n+l+l'-1-2i, r+r'+i))\\
&\oplus(\oplus_{i=c(t)}^{t}msP(l+l'-1-2i, r+r'+i))\\
&\oplus(\oplus_{i=c(l+l'-1)}^{l-1}msP(n+l+l'-1-2i, r+r'+i)).\\
\end{array}$$
$(2)$ If $m$ is even, then
$$\begin{array}{rl}
&\O^{-m}V(l, r)\ot M_s(l', r', \eta)\\
\cong&(\oplus_{i=t+1}^{l'-1}M_s(l+l'-1-2i, r+r'+i, \eta q^{2i-l+1}\frac{(l+l'-1-2i)_q}{(l')_q}))\\
&\oplus(\oplus_{i=t+1}^{l'-1}msP(l+l'-1-2i, r+r'+i))\\
&\oplus(\oplus_{i=c(t)}^{t}(m+1)sP(l+l'-1-2i, r+r'+i))\\
&\oplus(\oplus_{i=c(l+l'-1)}^{l-1}(m+1)sP(n+l+l'-1-2i, r+r'+i)).\\
\end{array}$$
\end{corollary}

\begin{proof}
It is similar to Corollary \ref{5.7} by using the duality $(-)^*$,
Lemmas \ref{3.12}-\ref{3.13}, and Proposition \ref{5.12}.
\end{proof}

\subsection{Tensor product of two band modules}\selabel{5.3}
~~

In this subsection, we investigate the tensor product $M=M_m(l,r,\a)\ot M_s(l', r', \eta)$ of
two band modules. By \cite{EGST}, any non-projective indecomposable summand of $M$
must be a band module. The module on an example with $m=s=1$ is displayed in \cite{EGST} for the special case
$n=d=6$. We will determine $M$ for two cases $\a q^{1-l'}(l')_q\neq\eta q^{1-l}(l)_q$
and $\a q^{1-l'}(l')_q=\eta q^{1-l}(l)_q$, respectively.

In the case of $\a q^{1-l'}(l')_q\neq\eta q^{1-l}(l)_q$, we show that $M$ is projective, and
determine the decomposition of $M$ by the inductions on $m+s$ and $l+l'$. We first determine
$M$ for $l=l'=1$ by the induction on $m+s$. Here we use the last exact sequence in \seref{2}.
Then tensoring by $V(2,0)$, we determine $M$ by the induction on $l+l'$.

In the case of $\a q^{1-l'}(l')_q=\eta q^{1-l}(l)_q$, we first determine $M$ for $l+l'\<n$
by the induction on $l+l'$. For $l=l'=1$, we use the exact sequence
$$\begin{array}{rl}
0\ra M_m(1, r, \eta)\ot M_s(1, r', \eta)&\ra\O^mV\ot M_s(1, r', \eta)\\
&\xrightarrow{f} V(n-1, r+1)\ot M_s(1, r', \eta)\ra 0,\\
\end{array}$$
where $V=V(1, r)$ for $m$ being odd, and $V=V(n-l,r+l)$ for $m$ being even.
The decompositions of the middle and right terms are known. We show that the non-projective summand
of the middle term is contained in the kernel of $f$, which gives rise to the decomposition
of the left term by Lemma \ref{5.1}. For the induction step, we use tensoring with $V(2,0)$.
Finally, applying the duality $(-)^*$ to the decomposition
of $M$ for $l+l'<n$, one gets the decomposition of $M$ for $l+l'\>n$.

Now we first consider the case of $\a q^{1-l'}(l')_q\neq\eta q^{1-l}(l)_q$.

\begin{lemma}\label{5.14}
Let $\eta\in{\mathbb P}^1(k)$ and $M\in\mathcal M$.
Assume that $M$ fits into an exact sequence
$$0\ra M_1(1, 0, \eta)\ra M\ra M_1(n-1, 1, -\eta q)\ra 0.$$
Then $M\cong M_1(1, 0,\eta)\oplus M_1(n-1, 1, -\eta q)$ or $M\cong P(1, 0)$.
\end{lemma}

\begin{proof}
From the exact sequence $0\ra M_1(1, 0, \eta)\ra P(1, 0)\ra M_1(n-1, 1, -\eta q)\ra 0$, one gets a long
exact sequence
$$\begin{array}{rl}
0&\ra {\rm Hom}_{H_n(1, q)}(M_1(n-1, 1, -\eta q), M_1(1, 0, \eta))\\
&\ra {\rm Hom}_{H_n(1, q)}(P(1, 0), M_1(1, 0, \eta))
\ra {\rm Hom}_{H_n(1, q)}(M_1(1, 0, \eta), M_1(1, 0, \eta))\\
&\ra {\rm Ext}^1_{H_n(1, q)}(M_1(n-1, 1, -\eta q), M_1(1, 0, \eta))\ra 0.\\
\end{array}$$
A straightforward verification shows that ${\rm Hom}_{H_n(1, q)}(M_1(n-1, 1, -\eta q), M_1(1, 0, \eta))$,
${\rm Hom}_{H_n(1, q)}(P(1, 0), M_1(1, 0, \eta))$ and ${\rm Hom}_{H_n(1, q)}(M_1(1, 0, \eta), M_1(1, 0, \eta))$
are all one dimensional over $k$. Hence ${\rm Ext}^1_{H_n(1, q)}(M_1(n-1, 1, -\eta q), M_1(1, 0, \eta))\cong k$.
It follows that $M\cong M_1(1, 0,\eta)\oplus M_1(n-1, 1, -\eta q)$ or $M\cong P(1, 0)$.
\end{proof}

\begin{lemma}\label{5.15}
Let $r, r'\in{\mathbb Z}_n$, $\a, \eta\in{\mathbb P}^1(k)$ and $s, m\>1$. Assume
$\a\neq\eta$. Then
$$M_m(1, r,\a)\ot M_s(1, r', \eta)\cong\oplus_{i=1}^{c(n)}msP(2i-1, r+r'-i+1).$$
\end{lemma}

\begin{proof}
It is enough to show the lemma for $r=r'=0$. We prove it by the induction
on $m+s$. We first assume that $m+s=2$. Then $m=s=1$.
Let $M=M_1(1, 0,\a)\ot M_1(1, 0, \eta)$.
Applying $M_1(1, 0,\a)\ot$ to the exact sequence
$0\ra V(1, 0)\ra M_1(1, 0, \eta)\ra V(n-1, 1)\ra 0$, one gets another exact sequence
$0\ra M_1(1, 0,\a)\ot V(1, 0)\ra M\ra M_1(1, 0,\a)\ot V(n-1, 1)\ra 0$.
By Theorem \ref{3.16}, we have
$$\begin{array}{rl}
&M_1(1,0,\a)\ot V(n-1,1)\\
\cong& M_1(n-1, 1, -\a q))
\oplus(\oplus_{i=c(n-1)}^{n-2} P(2n-1-2i, i+1))\\
\cong& M_1(n-1, 1, -\a q)
\oplus(\oplus_{i=2}^{c(n)} P(2i-1, 1-i)).\\
\end{array}$$
By $M_1(1, 0,\a)\ot V(1, 0)\cong M_1(1, 0,\a)$, it follows that
there exist two submodules $M_1$ and $M_2$ of $M$ with $M=M_1\oplus M_2$ such that
$M_2\cong \oplus_{i=2}^{c(n)} P(2i-1, 1-i)$ and $M_1$ fits an exact sequence
$$0\ra M_1(1, 0, \a)\ra M_1\ra M_1(n-1, 1, -\a q)\ra 0.$$
Then by Lemma \ref{5.14}, $M_1\cong M_1(1, 0,\a)\oplus M_1(n-1, 1, -\a q)$ or $M_1\cong P(1, 0)$.
Since $M_1(1, 0,\a)\ot M_1(1, 0, \eta)\cong M_1(1, 0,\eta)\ot M_1(1, 0, \a)$,
a similar argument as above shows that $M=N_1\oplus N_2$, where $N_1$ and $N_2$ are submodules of $M$,
$N_2\cong \oplus_{i=2}^{c(n)} P(2i-1, 1-i)$, and
$N_1\cong M_1(1, 0,\eta)\oplus M_1(n-1, 1, -\eta q)$ or $N_1\cong P(1, 0)$.
Since $M=M_1\oplus M_2=N_1\oplus N_2$ and $M_2\cong N_2$,
it follows from Krull-Schmidt Theorem that $M_1\cong N_1$.
However, $M_1(1, 0,\a)\oplus M_1(n-1, 1, -\a q)\ncong M_1(1, 0,\eta)\oplus M_1(n-1, 1, -\eta q)$
by $\a\neq \eta$. Therefore, $M_1\cong N_1\cong P(1, 0)$.
Thus, we have $M_1(1, 0,\a)\ot M_1(1, 0, \eta)\cong\oplus_{i=1}^{c(n)} P(2i-1, 1-i)$.

Now assume that $m+s>2$. We may assume that $m\>2$ without losing the generality.
Then there is an exact sequence
$$\begin{array}{rl}
0\ra M_{m-1}(1, 0, \a)\ot M_s(1,0,\eta)&\ra M_{m}(1, 0, \a)\ot M_s(1,0,\eta)\\
&\ra M_1(1, 0, \a)\ot M_s(1,0,\eta)\ra 0.\\
\end{array}$$
By the induction hypothesis, $M_1(1, 0, \a)\ot M_s(1, 0, \eta)$ is projective, and so
the above exact sequence is split. Again by the induction hypothesis, we have
$$\begin{array}{rl}
&M_m(1, 0, \a)\ot M_s(1, 0, \eta)\\
\cong&M_{m-1}(1, 0, \a)\ot M_{s}(1, 0, \eta)\oplus M_1(1, 0, \a)\ot M_s(1, 0, \eta)\\
\cong&(\oplus_{i=1}^{c(n)}(m-1)sP(2i-1, 1-i))\oplus(\oplus_{i=1}^{c(n)}sP(2i-1, 1-i))\\
\cong&\oplus_{i=1}^{c(n)}msP(2i-1, 1-i).\\
\end{array}$$
This completes the proof.
\end{proof}

\begin{proposition}\label{5.16}
Let $1\<l\<l'<n$, $r, r'\in{\mathbb Z}_n$, $\a, \eta\in{\mathbb P}^1(k)$
and $s, m\>1$. Assume that $\a\neq\eta$. Then
$$\begin{array}{rl}
&M_m(l, r, \a q^{1-l}(l)_q)\ot M_s(l', r', \eta q^{1-l'}(l')_q)\\
\cong&(\oplus_{i=1}^{c(n+l-l')}msP(l'-l-1+2i, r+r'+l-i))\\
&\oplus(\oplus_{c(l+l'-1)\<i\<l'-1}msP(n+l+l'-1-2i,r+r'+i)).\\
\end{array}$$
\end{proposition}

\begin{proof}
It is enough to show the proposition for $r=r'=0$. We prove it by the induction
on $l+l'$. For $l+l'=2$, it follows from Lemma \ref{5.15}. Now let $l+l'>2$.

We only consider the case of $l<l'$ since the proof is similar for $l=l'$.
In this case, $l'-1\>l$. By the induction hypothesis, we have
$$\begin{array}{rl}
&M_m(l, 0, \a q^{1-l}(l)_q)\ot M_s(l'-1, 0, \eta q^{2-l'}(l'-1)_q)\ot V(2, 0)\\
\cong&(\oplus_{i=1}^{c(n+l-l'+1)}msP(l'-l-2+2i, l-i)\ot V(2, 0))\\
&\oplus(\oplus_{c(l+l'-2)\<i\<l'-2}msP(n+l+l'-2-2i,i)\ot V(2, 0)).\\
\end{array}$$
If $l'=2$, then $l=1$. Hence by Lemma \ref{3.11} and Theorem \ref{3.17}, we have
$$\begin{array}{rl}
&M_m(1, 0, \a)\ot M_s(1, 0, \eta)\ot V(2, 0)\\
\cong&M_m(1, 0, \a)\ot M_s(2, 0,\eta q^{-1}(2)_q)\oplus sM_m(1, 0, \a)\ot V(n, 1)\\
\cong&M_m(1, 0, \a)\ot M_s(2, 0,\eta q^{-1}(2)_q)
\oplus(\oplus_{i=1}^{c(n-1)}msP(2i, 1-i))\oplus msV(n, 1).\\
\end{array}$$
In this case, using Proposition \ref{3.1}, Theorems \ref{3.3} and \ref{3.5},
a straightforward computation (for $n$ to be even and odd, respectively) shows that
$$\oplus_{i=1}^{c(n)} P(2i-1, 1-i)\ot V(2, 0)\\
\cong(\oplus_{i=1}^{c(n-1)}2P(2i, 1-i))\oplus 2V(n, 1).$$
Thus, it follows from Krull-Schmidt Theorem that
$$M_m(1, 0, \a)\ot M_s(2, 0,\eta q^{-1}(2)_q)\cong
(\oplus_{i=1}^{c(n-1)}msP(2i, 1-i))\oplus msV(n, 1),$$
as desired. If $l'>2$ and $l\<l'-2$, then by Lemma \ref{3.15} (or Theorem \ref{3.16})
and the induction hypothesis, we have
$$\begin{array}{rl}
&M_m(l, 0, \a q^{1-l}(l)_q)\ot M_s(l'-1, 0, \eta q^{2-l'}(l'-1)_q)\ot V(2, 0)\\
\cong&M_m(l, 0, \a q^{1-l}(l)_q)\ot M_s(l', 0,\eta q^{1-l'}(l')_q)\\
&\oplus M_m(l, 0, \a q^{1-l}(l)_q)\ot M_s(l'-2, 1, \eta q^{3-l'}(l'-2)_q)\\
\cong&M_m(l, 0, \a q^{1-l}(l)_q)\ot M_s(l', 0,\eta q^{1-l'}(l')_q)\\
&\oplus(\oplus_{i=1}^{c(n+l-l'+2)}msP(l'-l-3+2i, 1+l-i))\\
&\oplus(\oplus_{c(l+l'-3)\<i\<l'-3}msP(n+l+l'-3-2i,1+i))\\
\cong&M_m(l, 0, \a q^{1-l}(l)_q)\ot M_s(l', 0,\eta q^{1-l'}(l')_q)\\
&\oplus(\oplus_{i=0}^{c(n+l-l')}msP(l'-l-1+2i, l-i))\\
&\oplus(\oplus_{c(l+l'-1)\<i\<l'-2}msP(n+l+l'-1-2i,i)).\\
\end{array}$$
On the other hand, by Proposition \ref{3.1}, Theorems \ref{3.3} and \ref{3.5},
one can check that
$$\begin{array}{rl}
&(\oplus_{i=1}^{c(n+l-l'+1)} P(l'-l-2+2i, l-i)\ot V(2, 0))\\
&\oplus(\oplus_{c(l+l'-2)\<i\<l'-2}P(n+l+l'-2-2i,i)\ot V(2, 0))\\
\cong&(\oplus_{i=1}^{c(n+l-l')} P(l'-l-1+2i, l-i))\\
&\oplus(\oplus_{i=0}^{c(n+l-l')} P(l'-l-1+2i, l-i))\\
&\oplus(\oplus_{i=c(l+l'-1)}^{l'-1}P(n+l+l'-1-2i,i))\\
&\oplus(\oplus_{c(l+l'-1)\<i\<l'-2}P(n+l+l'-1-2i,i)).\\
\end{array}$$
Hence it follows from Krull-Schmidt Theorem that
$$\begin{array}{rl}
&M_m(l, 0, \a q^{1-l}(l)_q)\ot M_s(l', 0, \eta q^{1-l'}(l')_q)\\
\cong&(\oplus_{i=1}^{c(n+l-l')}msP(l'-l-1+2i, l-i))\\
&\oplus(\oplus_{i=c(l+l'-1)}^{l'-1}msP(n+l+l'-1-2i, i)),\\
\end{array}$$
as desired. If $l'>2$ and $l=l'-1$, then by the induction hypothesis, and Theorems \ref{3.3}
and \ref{3.5}, a similar argument as above shows that
$$M_m(l, 0, \a q^{1-l}(l)_q)\ot M_s(l', 0,\eta q^{1-l'}(l')_q)
\cong(\oplus_{i=1}^{c(n-1)}msP(2i, l-i))\oplus msV(n, l).$$
This completes the proof.
\end{proof}

\begin{corollary}\label{5.17}
Let $1\<l\<l'<n$, $r, r'\in{\mathbb Z}_n$, $\a, \eta\in{\mathbb P}^1(k)$
and $s, m\>1$. Assume that $\a q^{1-l'}(l')_q\neq\eta q^{1-l}(l)_q$. Then
$$\begin{array}{rl}
&M_m(l, r, \a)\ot M_s(l', r', \eta)\\
\cong&(\oplus_{i=1}^{c(n-l'+l)}msP(l'-l-1+2i, r+r'+l-i))\\
&\oplus(\oplus_{c(l+l'-1)\<i\<l'-1}msP(n+l+l'-1-2i,r+r'+i)).\\
\end{array}$$
\end{corollary}

\begin{proof}
It follows from Proposition \ref{5.16}.
\end{proof}

Now we investigate $M_m(l, r, \a)\ot M_s(l', r', \eta)$ for $\a q^{1-l'}(l')_q=\eta q^{1-l}(l)_q$.
We only need to consider the case $m\>s$ since
$M_m(l,r,\a)\ot M_s(l', r', \eta)\cong M_s(l', r', \eta)\ot M_m(l,r,\a)$.

\begin{lemma}\label{5.18}
Let $M$ be an indecomposable module with ${\rm rl}(M)=2$.\\
$(1)$ If $M$ is of $(s+1, s)$-type, then $M$ contains no submodules
of $(i+1, i)$-type for any $s>i\>1$, and consequently, $M$ contains no proper submodule
$N$ with $l(N/{\rm soc}(N))>l(N)$.\\
$(2)$ If $M$ is of $(s, s)$-type, then $M$ contains no submodules
of $(i+1, i)$-type, and consequently, $M$ contains no submodule
$N$ with $l(N/{\rm soc}(N))>l(N)$.
\end{lemma}

\begin{proof}
It follows from \cite[Lemma 4.3]{Ch3} and \cite[Proposition 3.3]{Ch4}. It also can be shown by an
argument similar to the proof of \cite[Lemma 4.3]{Ch3}.
\end{proof}

\begin{lemma}\label{5.19}
Let $s\>1$ and $M$ be an indecomposable module of $(s, s)$-type. Then $M$ can be embedded into an indecomposable
module of $(s+1, s)$-type.
\end{lemma}

\begin{proof}
It is similar to \cite[Lemma 3.28]{Ch5} by using Lemma \ref{5.18}.
\end{proof}

\begin{lemma}\label{5.20}
Let $r, r'\in{\mathbb Z}_n$, $\eta\in{\mathbb P}^1(k)$ and $s\>1$. Then
$M_s(1, r, \eta)\ot M_s(1, r', \eta)$ contains a submodule isomorphic to
$M_s(n-1, r+r'+1, -\eta q)$.
\end{lemma}

\begin{proof}
By Lemma \ref{3.8}, it is enough to show the lemma for $r=r'=0$.
Assume that $\eta\in k$ and let $M=M_s(1, 0, \eta)\ot M_s(1, 0, \eta)$.
By Lemma \ref{3.10}, there is a standard basis $\{v_{i,j}|1\<i\<n, 1\<j\<s\}$
in $M_s(1,0,\eta)$ such that
\begin{equation*}
\begin{array}{ll}
av_{i,j}=\left\{\begin{array}{ll}
v_{i+1,j} , & 1\<i<n,\\
0 , & i=n,\\
\end{array}\right. &
bv_{i,j}=q^iv_{i,j},\\
dv_{i,j}=\left\{\begin{array}{ll}
v_{n,j-1}+\eta qv_{n,j}, & i=1,\\
\a_{i-1}(n-1)v_{i-1,j}, & 1<i\<n-1,\\
0, & i=n,\\
\end{array}\right.& cv_{i,j}=q^iv_{i,j},\\
\end{array}
\end{equation*}
where $1\<i\<n$, $1\<j\<s$ and $v_{n,0}=0$.
Then $\{v_{i,j}\ot v_{l,m}|1\<i, l\<n, 1\<j, m\<s\}$ is a basis of $M$.
For $1\<i\<n$ and $1\<j\<s$, let $u_{i,j}\in M$ be defined by
$$u_{1,j}=(-q)^j\sum_{l=1}^j\sum_{m=1}^{n-1}(-1)^{m-1}q^{-\frac{m(m-1)}{2}}v_{m,l}\ot v_{n-m,j+1-l}$$
and
$$u_{i,j}=(-q)^j\sum_{l=1}^j(v_{i-1,l}\ot v_{n,j+1-l}-v_{n,l}\ot v_{i-1,j+1-l})$$
for $2\<i\<n$. Then $\{u_{i,j}|1\<i\<n, 1\<j\<s\}$ are linearly independent over $k$.
A tedious but standard verification shows that $N={\rm span}\{u_{i,j}|1\<i\<n, 1\<j\<s\}$
is a submodule of $M$, and $N\cong M_s(n-1, 1, -\eta q)$ by Lemma \ref{3.10}.

For $\eta=\oo$, using Lemma \ref{3.9}, one can similarly show that $M_s(1, 0, \oo)\ot M_s(1, 0, \oo)$
contains a submodule isomorphic to $M_s(n-1, 1, \oo)$.
This completes the proof.
\end{proof}

\begin{lemma}\label{5.21}
Let $r, r'\in{\mathbb Z}_n$, $\eta\in{\mathbb P}^1(k)$ and $m\>s\>1$. Then
$$\begin{array}{rl}
&M_m(1, r, \eta)\ot M_s(1, r', \eta)\\
\cong&M_s(1, r+r', \eta)\oplus M_s(n-1, r+r'+1, -\eta q)\\
&\oplus(m-1)sP(1, r+r')\oplus(\oplus_{i=1}^{c(n-2)}msP(2i+1, r+r'-i)).\\
\end{array}$$
\end{lemma}

\begin{proof}
It is enough to show the lemma for $r=r'=0$.
We only consider the case that $m$ is odd since the proof is similar for $m$ being even.

Assume that $m$ is odd. Then by Lemma \ref{5.19}, there is an exact sequence
$$0\ra M_m(1, 0, \eta)\ra \O^mV(1,0)\ra V(n-1, 1)\ra 0.$$
Applying $\ot M_s(1, 0, \eta)$ to the above sequence, one gets the following exact sequence
$$\begin{array}{rl}
0\ra M_m(1, 0, \eta)\ot M_s(1, 0, \eta)&\xrightarrow{\s}\O^mV(1, 0)\ot M_s(1, 0, \eta)\\
&\ra V(n-1, 1)\ot M_s(1, 0, \eta)\ra 0.\\
\end{array}$$
By \cite[Theorem 3.10(2)]{Ch4}, $M_m(1, 0, \eta)$ contains a unique submodule $M$ of $(s, s)$-type, and
$M\cong M_s(1, 0, \eta)$. From Lemma \ref{5.20}, one knows that $M\ot M_s(1, 0, \eta)$
contains a submodule isomorphic to $M_s(n-1, 1, -\eta q)$. It follows that $M_m(1, 0, \eta)\ot M_s(1, 0, \eta)$
contains a submodule $N$ isomorphic to $M_s(n-1, 1, -\eta q)$. From Proposition \ref{5.6},
$\O^mV(1, 0)\ot M_s(1, 0, \eta)$ contains submodules $M'$ and $P$ with $M'\cong M_s(n-1, 1, -\eta q)$ and
$P\cong msP(1, 0)\oplus(\oplus_{i=1}^{c(n-2)}(m+1)sP(2i+1, -i))$
such that $\O^mV(1, 0)\ot M_s(1, 0, \eta)=P\oplus M'$. Since $\s$ is a monomorphism,
$\s(N)\cong N\cong M_s(n-1, 1, -\eta q)$, and hence ${\rm soc}(\s(N))\cong sV(n-1, 1)$.
However, ${\rm soc}(P)\cong msV(1, 0)\oplus(\oplus_{i=1}^{c(n-2)}(m+1)sV(2i+1, -i))$
since ${\rm soc}(P(l, r))\cong V(l, r)$ for all $1\<l\<n$ and $r\in\mathbb{Z}$.
It follows that the sum $P+\s(N)$ is direct, and so $\O^mV(1, 0)\ot M_s(1, 0, \eta)=P\oplus M'=P\oplus \s(N)$ by
comparing their lengths. Hence
we have the following exact sequence
$$0\ra M_m(1, 0, \eta)\ot M_s(1, 0, \eta)\xrightarrow{\s} P\oplus \s(N)
\xrightarrow{f}V(n-1, 1)\ot M_s(1, 0, \eta)\ra 0.$$
Since $f$ is an epimorphism and $f(\s(N))=0$, $f|_P: P\ra V(n-1, 1)\ot M_s(1, 0, \eta)$
is an epimorphism. By Theorem \ref{3.16}, we have
$$\begin{array}{rl}
&V(n-1, 1)\ot M_s(1, 0, \eta)\\
\cong& M_s(n-1, 1, -\eta q)\oplus(\oplus_{i=c(n-1)}^{n-2}sP(2n-1-2i, i+1))\\
\cong& M_s(n-1, 1, -\eta q)\oplus(\oplus_{i=1}^{c(n-2)}sP(2i+1, -i)).\\
\end{array}$$
Note that $sP(1,0)\cong P(M_s(n-1,1,-\eta q))$ and $\O M_s(n-1,1,-\eta q)\cong M_s(1, 0, \eta)$.
Hence we have
$$P\cong P(V(n-1, 1)\ot M_s(1, 0, \eta))\oplus(m-1)sP(1, 0)\oplus(\oplus_{i=1}^{c(n-2)}msP(2i+1, -i)).$$
It follows from Lemma \ref{5.1} that
$${\rm Ker}(f|_P)\cong M_s(1, 0, \eta)\oplus(m-1)sP(1, 0)\oplus(\oplus_{i=1}^{c(n-2)}msP(2i+1, -i)),$$
and so
$$\begin{array}{rl}
&M_m(1, 0, \eta)\ot M_s(1, 0, \eta)\cong {\rm Ker}(f)={\rm Ker}(f|_P)\oplus\s(N)\\
\cong&M_s(1, 0, \eta)\oplus M_s(n-1, 1, -\eta q)\\
&\oplus(m-1)sP(1, 0)\oplus(\oplus_{i=1}^{c(n-2)}msP(2i+1, -i)).\\
\end{array}$$
\end{proof}

\begin{theorem}\label{5.22}
Let $1\<l, l'<n$, $r, r'\in{\mathbb Z}_n$, $\eta\in{\mathbb P}^1(k)$ and $m\>s\>1$.
Assume that $l+l'\<n$, and let $l_1={\rm min}\{l, l'\}$ and $l_2={\rm max}\{l, l'\}$. Then
$$\begin{array}{rl}
&M_m(l, r, \eta q^{1-l}(l)_q)\ot M_s(l', r', \eta q^{1-l'}(l')_q)\\
\cong&(\oplus_{i=0}^{l_1-1}M_s(l+l'-1-2i, r+r'+i, \eta q^{2i-l-l'+2}(l+l'-1-2i)_q))\\
&\oplus(\oplus_{i=l_2}^{l+l'-1}M_s(n+l+l'-1-2i, r+r'+i, -\eta q(2i-l-l'+1)_q))\\
&\oplus(\oplus_{i=0}^{l_1-1}(m-1)sP(l+l'-1-2i, r+r'+i))\\
&\oplus(\oplus_{1\<i\<c(n-l-l')}msP(l+l'-1+2i, r+r'-i))\\
&\oplus(\oplus_{c(l+l'-1)\<i\<l_2-1}msP(n+l+l'-1-2i, r+r'+i)).\\
\end{array}$$
\end{theorem}

\begin{proof}
It is enough to show the theorem for $r=r'=0$. We prove it by the induction on $l+l'$.
For $l+l'=2$, it follows from Lemma \ref{5.21}.
Now assume that $l+l'>2$. Here we only consider the case of $l=l'$ since the proof are similar
for the other cases: $l<l'-1$, $l=l'-1$, $l>l'+1$ and $l=l'+1$.
Suppose $l=l'$. Then $l\>2$. By the induction hypothesis,
applying $V(2,0)\ot$ and then using Proposition \ref{3.1}, Theorems \ref{3.3},
\ref{3.5} and \ref{3.16}, a tedious but standard computation shows that
$$\begin{array}{rl}
&V(2, 0)\ot M_m(l-1, 0, \eta q^{2-l}(l-1)_q)\ot M_s(l, 0, \eta q^{1-l}(l)_q)\\
\cong&(\oplus_{i=0}^{l-2}(V(2, 0)\ot M_s(2l-2-2i, i, \eta q^{2i-2l+3}(2l-2-2i)_q))\\
&\oplus(\oplus_{i=l}^{2l-2}V(2, 0)\ot M_s(n+2l-2-2i, i, -\eta q(2i-2l+2)_q))\\
&\oplus(\oplus_{i=0}^{l-2}(m-1)sV(2, 0)\ot P(2l-2-2i, i))\\
&\oplus(\oplus_{i=1}^{c(n-2l+1)}msV(2, 0)\ot P(2l-2+2i, -i))\oplus msV(2, 0)\ot V(n, l-1)\\
\cong&(\oplus_{i=0}^{l-2}M_s(2l-1-2i, i, \eta q^{2i-2l+2}(2l-2i-1)_q))\\
&\oplus(\oplus_{i=1}^{l-1}M_s(2l-1-2i, i, \eta q^{2i-2l+2}(2l-2i-1)_q))\\
&\oplus(\oplus_{i=l}^{2l-2}M_s(n+2l-1-2i, i, -\eta q(2i-2l+1)_q))\\
&\oplus(\oplus_{i=l+1}^{2l-1}M_s(n+2l-1-2i, i, -\eta q(2i-2l+1)_q))\\
&\oplus(\oplus_{i=0}^{l-2}(m-1)sP(2l-1-2i, i))\oplus(\oplus_{i=1}^{l-1}(m-1)sP(2l-1-2i,i))\\
&\oplus(\oplus_{1\<i\<c(n-2l)}msP(2l-1+2i, -i))\\
&\oplus(\oplus_{i=0}^{c(n-2l)}msP(2l-1+2i, -i))\oplus msP(n-1, l).\\
\end{array}$$
If $l>2$, then by Theorem \ref{3.16} and the induction hypothesis, we have
$$\begin{array}{rl}
&V(2, 0)\ot M_m(l-1, 0, \eta q^{2-l}(l-1)_q)\ot M_s(l, 0, \eta q^{1-l}(l)_q)\\
\cong&M_m(l, 0, \eta q^{1-l}(l)_q)\ot M_s(l, 0, \eta q^{1-l}(l)_q)\\
&\oplus M_m(l-2, 1, \eta q^{3-l}(l-2)_q)\ot M_s(l, 0, \eta q^{1-l}(l)_q)\\
\cong&M_m(l, 0, \eta q^{1-l}(l)_q)\ot M_s(l, 0, \eta q^{1-l}(l)_q)\\
&\oplus(\oplus_{i=0}^{l-3}M_s(2l-3-2i, i+1, \eta q^{2i-2l+4}(2l-3-2i)_q))\\
&\oplus(\oplus_{i=l}^{2l-3}M_s(n+2l-3-2i, i+1, -\eta q(2i-2l+3)_q))\\
&\oplus(\oplus_{i=0}^{l-3}(m-1)sP(2l-3-2i, i+1))\\
&\oplus(\oplus_{i=1}^{c(n-2l+2)}msP(2l-3+2i, 1-i))\oplus msP(n-1, l))\\
\cong&M_m(l, 0, \eta q^{1-l}(l)_q)\ot M_s(l, 0, \eta q^{1-l}(l)_q)\\
&\oplus(\oplus_{1\<i\<l-2}M_s(2l-1-2i, i, \eta q^{2i-2l+2}(2l-1-2i)_q))\\
&\oplus(\oplus_{l+1\<i\<2l-2}M_s(n+2l-1-2i, i, -\eta q(2i-2l+1)_q))\\
&\oplus(\oplus_{1\<i\<l-2}(m-1)sP(2l-1-2i, i))\\
&\oplus(\oplus_{i=0}^{c(n-2l)}msP(2l-1+2i, -i))\oplus msP(n-1, l)).\\
\end{array}$$
If $l=2$, by Lemma \ref{3.11} and Theorem \ref{3.17}, one can similarly show
the above isomorphism.
Hence by Krull-Schmidt Theorem, we have
$$\begin{array}{rl}
&M_m(l, 0, \eta q^{1-l}(l)_q)\ot M_s(l, 0, \eta q^{1-l}(l)_q)\\
\cong&(\oplus_{i=0}^{l-1}M_s(2l-1-2i, i, \eta q^{2i-2l+2}(2l-2i-1)_q))\\
&\oplus(\oplus_{i=l}^{2l-1}M_s(n+2l-1-2i, i, -\eta q(2i-2l+1)_q))\\
&\oplus(\oplus_{i=0}^{l-1}(m-1)sP(2l-1-2i, i))\\
&\oplus(\oplus_{1\<i\<c(n-2l)}msP(2l-1+2i, -i)).\\
\end{array}$$
This completes the proof.
\end{proof}

\begin{corollary}\label{5.23}
Let $1\<l, l'<n$, $r, r'\in{\mathbb Z}_n$, $\a, \eta\in{\mathbb P}^1(k)$ and $m\>s\>1$.
Assume that $l+l'\<n$, and let $l_1={\rm min}\{l, l'\}$ and $l_2={\rm max}\{l, l'\}$.
If $\a  q^{1-l'}(l')_q=\eta q^{1-l}(l)_q$, then
$$\begin{array}{rl}
&M_m(l, r, \a)\ot M_s(l', r', \eta)\\
\cong&(\oplus_{i=0}^{l_1-1}M_s(l+l'-1-2i, r+r'+i, \eta q^{2i-l+1}\frac{(l+l'-1-2i)_q}{(l')_q}))\\
&\oplus(\oplus_{i=l_2}^{l+l'-1} M_s(n+l+l'-1-2i, r+r'+i, -\eta q^{l'}\frac{(2i-l-l'+1)_q)}{(l')_q}))\\
&\oplus(\oplus_{i=0}^{l_1-1}(m-1)sP(l+l'-1-2i, r+r'+i))\\
&\oplus(\oplus_{1\<i\<c(n-l-l')}msP(l+l'-1+2i, r+r'-i))\\
&\oplus(\oplus_{c(l+l'-1)\<i\<l_2-1}msP(n+l+l'-1-2i, r+r'+i)).\\
\end{array}$$
\end{corollary}

\begin{proof}
It follows from Theorem \ref{5.22}.
\end{proof}

\begin{corollary}\label{5.24}
Let $1\<l, l'<n$, $r, r'\in{\mathbb Z}_n$, $\eta\in{\mathbb P}^1(k)$ and $m\>s\>1$.
Assume that $t=l+l'-(n+1)\>0$. Let $l_1={\rm min}\{l, l'\}$ and $l_2={\rm max}\{l, l'\}$. Then
$$\begin{array}{rl}
&M_m(l, r, \eta q^{1-l}(l)_q)\ot M_s(l', r', \eta q^{1-l'}(l')_q)\\
\cong&(\oplus_{i=t+1}^{l_1-1}M_s(l+l'-1-2i, r+r'+i, \eta q^{2i-l-l'+2}(l+l'-1-2i)_q))\\
&\oplus(\oplus_{i=l_2}^{n-1}M_s(n+l+l'-1-2i, r+r'+i, -\eta q(2i-l-l'+1)_q))\\
&\oplus(\oplus_{i=t+1}^{l_1-1}(m-1)sP(l+l'-1-2i, r+r'+i))\\
&\oplus(\oplus_{i=c(t)}^{t}msP(l+l'-1-2i, r+r'+i))\\
&\oplus(\oplus_{c(l+l'-1)\<i\<l_2-1}msP(n+l+l'-1-2i, r+r'+i)).\\
\end{array}$$
\end{corollary}

\begin{proof}
It is enough to show the corollary for $r=r'=0$. Since $1\<l, l'<n$ and $l+l'>n$,
$1\<n-l, n-l'<n$ and $(n-l)+(n-l')<n$. Hence by Theorem \ref{5.22}, we have
$$\begin{array}{rl}
&M_m(n-l, 1, -\eta q(l)_q)\ot M_s(n-l', 1, -\eta q(l')_q)\\
\cong&M_m(n-l, 1, \eta q^{1+l}(n-l)_q)\ot M_s(l', 1, \eta q^{1+l'}(n-l')_q)\\
\cong&(\oplus_{i=0}^{n-l_2-1}M_s(2n-l-l'-1-2i, i+2, \eta q^{2i+l+l'+2}(2n-l-l'-1-2i)_q))\\
&\oplus(\oplus_{i=n-l_1}^{2n-l-l'-1}M_s(3n-l-l'-1-2i, i+2, -\eta q(2i-2n+l+l'+1)_q))\\
&\oplus(\oplus_{i=0}^{n-l_2-1}(m-1)sP(2n-l-l'-1-2i, 2+i))\\
&\oplus(\oplus_{i=1}^{c(l+l'-n)}msP(2n-l-l'-1+2i, 2-i))\\
&\oplus(\oplus_{c(2n-l-l'-1)\<i\<n-l_1-1}msP(3n-l-l'-1-2i, 2+i)).\\
\end{array}$$
Then by applying the duality $(-)^*$ to the above isomorphism,
the corollary follows from Lemmas \ref{3.12} and \ref{3.13}.
\end{proof}

\begin{corollary}\label{5.25}
Let $1\<l, l'<n$, $r, r'\in{\mathbb Z}_n$, $\a, \eta\in{\mathbb P}^1(k)$ and $m\>s\>1$.
Assume that $t=l+l'-(n+1)\>0$. Let $l_1={\rm min}\{l, l'\}$ and $l_2={\rm max}\{l, l'\}$.
If $\a  q^{1-l'}(l')_q=\eta q^{1-l}(l)_q$, then
$$\begin{array}{rl}
&M_m(l, r, \a)\ot M_s(l', r', \eta)\\
\cong&(\oplus_{i=t+1}^{l_1-1}M_s(l+l'-1-2i, r+r'+i, \eta q^{2i-l+1}\frac{(l+l'-1-2i)_q}{(l')_q}))\\
&\oplus(\oplus_{i=l_2}^{n-1}M_s(n+l+l'-1-2i, r+r'+i, -\eta q^{l'}\frac{(2i-l-l'+1)_q}{(l')_q}))\\
&\oplus(\oplus_{i=t+1}^{l_1-1}(m-1)sP(l+l'-1-2i, r+r'+i))\\
&\oplus(\oplus_{i=c(t)}^{t}msP(l+l'-1-2i, r+r'+i))\\
&\oplus(\oplus_{c(l+l'-1)\<i\<l_2-1}msP(n+l+l'-1-2i, r+r'+i)).\\
\end{array}$$
\end{corollary}

\begin{proof}
It follows from Corollary \ref{5.24}.
\end{proof}

\centerline{ACKNOWLEDGMENTS}

The authors would like to show their gratitude to the referee for his/her helpful suggestions.
This work is supported by NSF of China (No. 11571298).\\

\end{document}